\documentclass[10pt]{article}
\usepackage{amsmath}
\usepackage{amsfonts}
\usepackage{mathrsfs}
\usepackage{bbm}
\usepackage{amsfonts}
\usepackage{enumitem}
\usepackage{amssymb,amsmath,amsthm,graphicx}
\usepackage{float} \usepackage[colorlinks=true]{hyperref}
 
\hypersetup{urlcolor=blue, citecolor=red}
\def\Xint#1{\mathchoice
{\XXint\displaystyle\textstyle{#1}}%
{\XXint\textstyle\scriptstyle{#1}}%
{\XXint\scriptstyle\scriptscriptstyle{#1}}%
{\XXint\scriptscriptstyle\scriptscriptstyle{#1}}%
\!\int}
\def\XXint#1#2#3{{\setbox0=\hbox{$#1{#2#3}{\int}$ }
\vcenter{\hbox{$#2#3$ }}\kern-.6\wd0}}

\def\dashint{\Xint-}

\allowdisplaybreaks

\newcommand{\be}{\begin{equation} \label}
\newcommand{\ee}{\end{equation}}
\newcommand{\bea}{\begin{eqnarray}\label}
\newcommand{\eea}{\end{eqnarray}}
\newcommand{\bas}{\begin{eqnarray*}}
\newcommand{\eas}{\end{eqnarray*}}
\newcommand{\bit}{\begin{itemize}}
\newcommand{\eit}{\end{itemize}}

\newcommand{\Om}{\Omega}

\newcommand{\lbal}{\left\{ \begin{array}{l}}
\newcommand{\lball}{\left\{ \begin{array}{ll}}
\newcommand{\ear}{\end{array} \right.}

\begin{document}
\title{Quantitative analysis and its applications for Keller-Segel type systems
\thanks{Supported by the National Natural
 Science Foundation of China (Nos. 12071098 and 12301245), Postdoctoral Science Foundation of Heilongjiang Province (No. LBH-Z22177) and the Young talents sponsorship program of Heilongjiang Province (No. 2023QNTJ004)}}
\author{Mengyao Ding$^{1}$,~~Yuzhou Fang$^{2}$,~~Chao Zhang$^{2,1}$\thanks{Corresponding author.
E-Mail:
mengyaod@126.com (M. Ding),
18b912036@hit.edu.cn (Y. Fang),
 czhangmath@hit.edu.cn (C. Zhang)}\\
 {\small $^{1}$ Institute for Advanced Study in Mathematics, Harbin Institute of Technology, Harbin 150001, PR China}\\
{\small $^{2}$ School of Mathematics, Harbin Institute of Technology, Harbin 150001, PR China}\\
}
\date{}
\newtheorem{theorem}{Theorem}
\newtheorem{definition}{Definition}[section]
\newtheorem{lemma}{Lemma}[section]
\newtheorem{proposition}{Proposition}[section]
\newtheorem{corollary}{Corollary}\newtheorem{remark}{Remark}
\renewcommand{\theequation}{\thesection.\arabic{equation}}
\catcode`@=11 \@addtoreset{equation}{section} \catcode`@=12
\maketitle{}

\begin{abstract}
In this paper, we utilize the De Giorgi iteration to quantitatively analyze the upper bound of solutions
for Keller-Segel type systems. The refined upper bound estimate presented here has broad applications in
determining large time behaviours of weak solutions and improving the regularity for models involving the $p$-Laplace operator.
To demonstrate the applicability of our findings, we investigate
the asymptotic stability of a chemotaxis model with nonlinear signal production and a chemotaxis-Navier-Stokes model with a logistic source. Additionally, within the context of $p$-Laplacian diffusion, we establish H\"{o}lder continuity for a chemotaxis-haptotaxis model and a chemotaxis-Stokes model.

\begin{description}
\item[2020MSC:] 35B45; 35B40; 35K55; 92C17
\item[Keywords:]Keller-Segel systems; De Giorgi iteration; Boundedness; Asymptotic stability
\end{description}
\end{abstract}
\section{Introduction}\label{INTR}
\setcounter{section}{1}\setcounter{equation}{0}

Chemotaxis is the phenomenon in which cells, organisms, or entities move in response to chemical stimulus.
Specially, the movements of single-cell or multicellular organisms are directed
by the concentration gradient of certain chemicals in their environment.
The pioneering modeling works, describing this biochemotactic phenomenon,
are introduced by Keller-Segel \cite{KS1970} and \cite{KS1971}.
Afterwards, given the crucial role of chemotaxis in the fields of medicine and biology,
an abundance of mathematical research on classical Keller-Segel systems
and plenty of related variants have been carried out extensively, see \cite{BBTW2015, Horstmann, TW} for review literature.

In this work, we consider the following chemotaxis system of a general form: 
 \begin{eqnarray}\label{q1}
 \begin{cases}
    \begin{array}{llll}
       \displaystyle n_t+\tau\mathbf{u}\cdot\nabla n=\nabla\cdot\big(a(\nabla n, n,x,t)\nabla n-b(n,x,t)\nabla c\big)+f(n,x,t), &x\in\Omega,\,t>0,\\
        \displaystyle  c_{t}+\tau\mathbf{u}\cdot\nabla c=\Delta c -c+g(n,x,t), &x\in\Omega,\,t>0,\\
        \displaystyle a(\nabla n, n,x,t)\frac{\partial n}{\partial \nu}=\frac{\partial c}{\partial \nu}=0, &x\in\partial\Omega,\,t>0,\\
        \displaystyle (n,c)|_{t=0}=(n_0,c_0),&x\in\Omega,\\
    \end{array}
 \end{cases}
\end{eqnarray}
where $\tau\in\{0,1\}$, $\Omega\subseteq \Bbb{R}^N$ ($N\ge2$) is a bounded domain with smooth boundary and $\partial/\partial \nu$ denotes the derivative with respect to
the outer normal of $\partial\Omega$.
The density of organisms such as cells and bacteria is represented by $n$; the concentration of chemical signal is represented by $c$.
In the present work, we also consider scenarios where the environmental fluid, denoted by $\mathbf{u}$, needs to be included in the model.
The component $\mathbf{u}$ with the property $\nabla\cdot\mathbf{u} =0$ in the distributional sense, 
satisfies (Navier-) Stokes equations involving an external
gravitational forces generated by the aggregation of organisms.
Here, we suppose that $\mathbf{u}$ is a given vector function admitting a certain regularity property.
The terms $\nabla\cdot(a(\nabla n, n, x,t)\nabla n)$ and $\nabla\cdot(b( n, x,t)\nabla c)$
exhibit the (nonlinear) self-diffusion and chemotactic cross-diffusion mechanisms, respectively.
The cells-kinetics mechanism is characterized by the logistic function $f(n,x,t)$.
The production or consumption mechanism of $n$ with respect to
$c$ is exhibited by the function $g(n,x,t)$.

Since there are various biological mechanisms in experimental environments or real-world implementations,
the nonlinear terms included in the modeled chemotaxis systems will be presented in diverse forms.
We first introduce the possible choices of $a$ and $b$.
The approach in \cite{hillen_painter2002} captures one prominent situation that finite volumes of individuals cannot be neglected. In a particular framework developed in \cite{hillen_painter2002},
the nonlinear terms $a$ and $b$ are connected through the relationship
\begin{align}
a(\xi, s, x,t)=Q(s)-s Q^{\prime}(s) \qquad \mbox{and} \qquad b( s, x,t)=s Q(s), \qquad\xi\in\Bbb{R}^N,~ s \geq 0,~x\in\overline{\Omega},~t>0,
\end{align}
where $Q(s)$ with $s=n(x,t)$ represents the density-dependent probability for a cell to find space somewhere
in its current neighborhood.
The above expressions also imply that $a$ and $b$ can be widely selected through kinds of ways and
a very prototypical choice is determined as follows,
\begin{align}\label{volume}
 	a(\xi, s, x,t)=(s+1)^{\alpha}
	\qquad \mbox{and} \qquad
	b(s, x,t)=s(s+1)^{\beta-1},
	 \qquad\xi\in\Bbb{R}^N,~ s \geq 0,~x\in\overline{\Omega},~t>0.
\end{align}
The solvability results concerning the chemotaxis model with nonlinear terms like \eqref{volume} can be found in \cite{SS2006, WD2010, TW2012}.
Apart from the above volume filling effects, the $p$-Laplacian as a regularity operator originating from non-Newtonian mechanics \cite{DT1994},
nonlinear flow laws \cite{GR2003} and other physical phenomena, 
also plays an important role in system modeling.
When it comes to the $p$-Laplacian, the nonlinear diffusion function $a$ reads as
$$a(\xi, s, x,t)= |\xi|^{p-2}, \qquad\xi\in\Bbb{R}^N,~ s \geq 0,~x\in\overline{\Omega},~t>0,$$
and under this setting, the existence of weak solutions to systems has been discussed in \cite{TL2019}.

Finally, let us turn to review some typical choices of $f$ and $g$.
Taking into account disciplines like population dynamics, it is necessary to incorporate a proper logistic function into the chemotaxis model
to characterize the proliferation and death of organisms.
The specific form of $f$ depends on practical factors such as environmental capacity and growth rate.
The logistic damping effects generated by $f$ under the most prototypical choice $f(s)=r s-\mu s^2$ ($r \in \mathbb{R}$ and $\mu>0$) have
been investigated by \cite{W2014,W2010}.
For the findings concerned with source functions of generalized types,
we refer the readers to e.g., \cite{WMX2021,Wang2020}.

When the chemical signal is produced by cells, the most common selection of the function $g$ is given by
$$g(s, x,t)=s, \qquad  s \geq 0,~x\in\overline{\Omega},~t>0,$$
which indicates that the rate of production of the signal is directly proportional to the quantity of substance.
Accounting for saturation
effects at large densities, the process of signal production through cells no longer needs to depend
on the population density in a linear manner, and then the prototype in this case is determined by the choice
$$g(s, x,t)=s^\sigma, \qquad s \geq 0,~x\in\overline{\Omega},~t>0.$$
An another important selection of $g$ was originally exhibited in \cite{KS1971}, 
where the chemotaxis model was built on the phenomenon that the chemical signal is a kind of nutrient consumed by the organisms.
Under this mechanism, a typical choice of $g$ takes the form that
$$g(s,x,t)=-s^{\sigma}c(x,t)+c(x,t), \qquad s \geq 0,~x\in\overline{\Omega},~t>0,$$
where $\sigma$ measures the consumption rate between the signal and cells.
In the field of PDE research,  the vast majority of studies on the chemotaxis model of
type \eqref{q1} focus on addressing solvability or demonstrating the existence of blow-up phenomena.
The solvability results, along with the regularity inferred from the space related to solution concepts, 
are achieved by investigating the energy development of local solutions
in the most existing literature, see more details in e.g. \cite{BBTW2015, Horstmann}.\\[8pt]
{\bf Motivation and ideas of the present work.}
\emph{Different from the approach taken in manuscripts addressing solvability, here we mainly adopt the De Giorgi-Nash theory instead of energy methods to analyze the properties of solutions.
In other words, our emphasis is on estimating the size of the level set rather than studying the development of  $\int_{\Omega} n^q(\cdot, t)\,dx$.
An advantage of employing the De Giorgi iteration is its ability to intuitively exhibit a delicate quantitative estimate on the upper bound of solutions.}

As demonstrated above, the diverse mechanisms in the real world lead to varied choices of nonlinear terms in the model.
To encompass a wider range of scenarios,
 we only impose the mild assumptions \eqref{Sa}--\eqref{Sf} on each nonlinear term appearing in \eqref{q1}.
Therefore, our boundedness results, as shown in Theorems \ref{th1}\&\ref{th2},
can provide quantitative analysis on upper bounds for global solutions obtained in many existing literature.
This idea also can be found in \cite{W20241, W20242}, where the authors captured the properties for parabolic
equations of a general form and further utilized them to investigate the chemotaxis models.
\\[8pt]
\noindent{\bf Hypotheses.}
In this study, we investigate the properties of solutions defined on $\overline{\Omega}\times[0, T)$, where $T\in(0,\infty)$ or $T=\infty$.
We suppose that
\begin{align}\label{Sa}
a(\xi, s, x,t)\ge a_0 (s+1)^{\alpha}|\xi|^{p-2},  \qquad\xi\in\Bbb{R}^N,~ s \geq 0,~x\in\overline{\Omega},~t>0
\end{align}
and
\begin{align}\label{Sb}
b(s, x,t)\le b_0 (s+1)^{\beta}\quad \text{ as well as}\quad b(0,x,t)=0,  \qquad\xi\in\Bbb{R}^N,~ s \geq 0,~x\in\overline{\Omega},~t>0
\end{align}
with $a_0>0$, $b_0\ge0$, $\alpha,\beta\in \Bbb{R}$ and $p>1$.
For any fixed $(x,t)\in\overline{\Omega}\times(0,T)$,
one-variable functions $f(\cdot,x,t)$ and $g(\cdot,x,t)$ belong to $ C^1([0,\infty))$.
Moreover, $f$ also satisfies that
\begin{align}\label{Sf0}
f(0, x, t)\ge0 \quad  \text{for all}\  \  (x, t)\in\overline{\Omega}\times(0, T)
\end{align}
and
\begin{align}\label{Sf}
  f(s,x,t)\rightarrow -\infty \quad \textmd{as } s\rightarrow \infty
\end{align}
uniformly in $\overline{\Omega}\times(0,T)$.
The condition \eqref{Sf} guarantees the existence of $K_f>1$ such that
\begin{align}\label{Kf}
f(s,x,t)\le 0\qquad {\rm for~any}~(x,t)\in\overline{\Omega}\times(0,T),
\end{align}
as long as $s\ge K_f$.
The velocity $\mathbf{u}$ is assumed to satisfy that
\begin{align}\label{mu}
\mathbf{u} \in L_{loc}^\infty(\overline{\Omega}\times[0,T))\qquad {\rm for }~ \tau=1.
\end{align}

\noindent{\bf Notations.}
To simplify presentations, we need to keep the coming notations in mind.
For any $\alpha\in\mathbb{R}$, the symbols $\alpha_-$ and $\alpha_+$ denote
$$\alpha_-:=\max\{-\alpha\, ,\, 0\},\quad\alpha_+:=\max\{\alpha\, ,\, 0\}.$$
As usual, we employ the domain notation $Q_T:=\overline{\Omega}\times(0, T)$.
In the following, the symbol $\frac{1}{\alpha_+}$ will appear several times,
which means 
\begin{align}
	\frac{1}{\alpha_+}:= \left\{ \begin{array}{ll}
	+\infty
	\qquad & \mbox{if } \alpha\le 0, \\[1mm]
	\frac{1}{\alpha}
	\qquad & \mbox{if } \alpha>0.
	\end{array} \right.
\end{align}

For any $k\in\Bbb{R}$ and $\tau\in(0,T)$, we define the level set
\begin{equation}\label{A}
A^{+}(k,\tau):=\left\{x \in \Omega\, \big| \,n(x, \tau)>k\right\}.
\end{equation}
For $h \in L^{1}(V)$, the mean average of $h$ is given by
$$
(h)_{V} :=\dashint_{V} h(x) \,d x=\frac{1}{|V|} \int_{V} h(x) \,dx.
$$ 
The presence of the fluid term $\mathbf{u}$ necessitates the introduction of the solenoidal subspace of $L^2(\Omega;\mathbb{R}^2)$ as follows,
$$L_{\sigma}^{2}(\Omega):=\big\{\mathbf{v}\in L^{2}(\Omega;\mathbb{R}^2) \mid \nabla \cdot \mathbf{v}=0\text { in the distributional sense} \big\}.$$

Throughout this paper, we adopt the shorthand notation
$$\verb"data":= (p, \alpha, \beta, a_0, b_0, N, \Omega).$$
Hereafter, $C\equiv C(\verb"data")$ represents a pure constant depending only on a subset of $\{p, \alpha, \beta, a_0, b_0, N, \Omega\}$, unless otherwise specified.

Before stating our main contributions, we introduce the concept of solutions in this paper.

\begin{definition}\label{def}
Let $T\in(0,\infty]$ and $\Omega\subseteq\mathbb{R}^N$ $(N\ge 2)$ be a bounded domain with smooth boundary.
Assume $a,b,f,g$ satisfy \eqref{Sa}--\eqref{Sf} and
$\mathbf{u}$ belongs to $L_{loc}^\infty(\overline{\Omega}\times[0,T))$
with the property that $\nabla\cdot\mathbf{ u}=0$ in the distributional sense.
We call a pair $(n,c)$ with $n\ge0$ in $Q_T$ a global weak solution of \eqref{q1} if
\begin{equation}
\begin{cases}
n \in L_{loc}^{p}\big((0,T) ; W^{1,p}(\Omega)\big) \cap C^{0}\big((0,T);L^{2+\alpha_-}(\Omega)\big),\\
c\in L_{l o c}^{2}\big((0,T) ; W^{1,2}(\Omega) \big)
\end{cases}
\end{equation}
such that
\begin{align}\label{da}
a\big(\nabla n(x,t), n(x,t), x,t\big)\big(1+n(x,t)\big)^{\alpha_-}|\nabla n(x,t)|
\in L_{l o c}^{\frac{p}{p-1}}\big(\overline{\Omega}\times(0,T)\big),
\end{align}
\begin{align}\label{db}
b\big(n(x,t),x,t\big)\big(1+n(x,t)\big)^{\alpha_-} |\nabla c(x,t)|
\in L_{l o c}^{\frac{p}{p-1}}\big(\overline{\Omega}\times(0,T)\big)
\end{align}
as well as
\begin{align}\label{df}
g(n(x,t),x,t)\in L_{l o c}^1\big(\overline{\Omega}\times(0,T)\big),\qquad f(n(x,t),x,t)\big(1+n(x,t)\big)^{1-\alpha_-}\in L_{l o c}^1\big(\overline{\Omega}\times(0,T)\big),
\end{align}
and that
\begin{equation}\label{De1}
\begin{split}
-\int_0^{\infty} \int_{\Omega}n \varphi_t\,dxdt
&=  \int_0^{\infty} \int_{\Omega}n \mathbf{u} \cdot \nabla \varphi\,dxdt
- \int_{0}^{\infty}\int_{\Omega} a(\nabla n, n, x,t)\nabla n \cdot \nabla\varphi\,dxdt \\
&\ \ \ + \int_0^{\infty}  \int_{\Omega}  b(n,x,t) \nabla c\cdot \nabla \varphi\,dxdt +  \int_0^{\infty} \int_{\Omega}f(n,x,t) \varphi\,dxdt
 \end{split}
\end{equation}
 and
\begin{equation}\label{De2}
-\int_0^\infty\int_{\Omega}c\psi_t\,dxdt
=-\int_0^\infty\int_{\Omega}\nabla c\cdot\nabla\psi-\int_0^\infty\int_{\Omega}c\psi\,dxdt
+\int_0^\infty\int_{\Omega}c\mathbf{u}\cdot\nabla \psi\,dxdt
+\int_0^\infty\int_{\Omega}g(n,x,t)\psi\,dxdt
\end{equation}
hold for any $\varphi\in C^{\infty}_0\big(\overline{\Omega}\times(0,T)\big)$ and $\psi\in C^{\infty}_0\big(\overline{\Omega}\times(0,T)\big)$.
\end{definition}

As stated in \cite[page 76]{D1993}, when considering a regularity operator of the type
$$
a(n,\nabla n)\simeq n^{\alpha}|\nabla n|^{p-2},
$$
it is crucial to emphasize the solution $n$ admits a property like \eqref{da} so that the integral term involving $a$ in \eqref{De1} makes sense, and subsequent choices of testing functions are feasible. Similarly, necessities of the requiring \eqref{db} and \eqref{df} arise from the same reason. The concept of weak solutions of this kind is very common in chemotaxis models, see e.g., \cite{W2015, TL2019, TW2007}.

Now we are in a position to state our main results.

\begin{theorem} [\label{th1}Boundedness result]
Let $T\in(0,\infty]$ and $\Omega\subseteq\mathbb{R}^N$ $(N\ge 2)$ be a bounded domain with smooth boundary.
Let a pair of functions $(n,c)$ with $n\ge 0$ in $Q_T$ 
be a weak solution to \eqref{q1} in the sense of Definition \ref{def},
where \eqref{Sa}--\eqref{Sf} are in force with
\begin{equation}\label{thm1-1}
\max\left\{ \frac{  p(\beta+\alpha_-) }{p-1}\,,\, 2+\alpha_-\right\}< \frac{p(N+2+\alpha_-)}{N}.
\end{equation}
Then for all $t_0\in(0,T)$ and any $\hat{t}\in(0,t_0) $, if there holds that
\begin{align}\label{tmathb}
\mathfrak{b}:=\max\Big\{\|\nabla c\|_{L^{ \infty}(\Omega\times(t_0-\hat{t},t_0))  }\, , \,1\Big\}<\infty,
\end{align}
we find that $n$ is bounded in $\Omega\times(t_0-\hat{t}/2,t_0)$.
Furthermore, the upper bound of $n$ is given as
\begin{align}\label{th11}
\underset{\Omega\times(t_0-\hat{t}/2,t_0)}
{\rm ess~ sup} n
\leq C  \mathfrak{b}^{ \kappa  }\left(\hat{t}+\frac{1}{  \hat{t}^{\frac{N}{p}}  } \right)^{    \hat{\kappa}    }
\left(\dashint_{t_0-\hat{t}}^{t_0}\int_{\Omega }n^{     \frac{p(N+2+\alpha_-)}{N}      }\,dxd\tau\right)^{    \hat{\kappa}       }+K_f,
\end{align}
where $K_f$ is determined in \eqref{Kf}, the positive pure constants $\kappa$, $\hat{\kappa}$,
 and $C$ depend only on \verb"data".
\end{theorem}
\begin{remark}\label{r1}
In the context of equation \eqref{q1}$_1$, the upper bound of the solution $n$ should depend on the norm of $\nabla c$ in a certain Lebesgue space, as well as on the
logistic function and parameters present in the equation. The display \eqref{th11} illustrates how these factors affect the sup-estimate of solutions in a detailed manner. 
By controlling these factors, we can achieve a desired $L^\infty$-bound of $n$, which can then be utilized to further explore other properties of solutions.
This is one of our intentions in providing a delicate estimate of the upper bound.
\end{remark}
\begin{remark}\label{r2}
Since the specific forms of $\kappa$ and $\hat{\kappa}$ are not closely relevant to our
subsequent contents, for simplicity, we merely provide a brief explanation of the dependency of
$\kappa$ and $\hat{\kappa}$ in the statement of Theorem \ref{th1}.
We remark that $\kappa$ and $\hat{\kappa}$ are precisely defined by
$$\frac{1}{\kappa}:=\frac{p-1}{p}  \bigg(\frac{p(N+2+\alpha_-)}{N}  -\max\left\{\frac{p(\beta+\alpha_-)N}{p-1}\,,\, 2+\alpha_-\,,\, p\right\}\bigg)$$
and
$$\frac{1}{\hat{\kappa}}:=\frac{p+N}{p}  \bigg(\frac{p(N+2+\alpha_-)}{N}  -\max\left\{\frac{p(\beta+\alpha_-)N}{p-1}\,,\, 2+\alpha_-\,,\, p\right\}\bigg).$$
\end{remark}

Unlike the scenario where $\frac{p(N+2+\alpha_-)}{N}>\max\Big\{ \frac{  p(\beta+\alpha_-) }{p-1}\,,\, 2+\alpha_-\Big\}$, when $p$ locates in the other range, the local boundedness is not implicit into the notion of weak solution and must be obtained by imposing other information.

\begin{theorem} [\label{th2}Boundedness result]
Let $T\in(0,\infty]$ and $\Omega\subseteq\mathbb{R}^N$ $(N\ge 2)$ be a bounded domain with smooth boundary.
Let a pair of functions $(n,c)$ with $n\ge0$ in $Q_T$ be a weak solution to \eqref{q1} in the sense of Definition \ref{def}, where \eqref{Sa}--\eqref{Sf} are in force with
\begin{equation}\label{thm2-1}
\max\left\{  \frac{p(\beta+\alpha_-)  }{p-1}\,,\, 2+\alpha_-\right\}\ge \frac{p(N+2+\alpha_-)}{N}.
\end{equation}
Then for all $t_0\in(0,T)$ and any $\hat{t}\in(0,t_0) $, if there holds that $n$ is bounded in $\Omega\times(t_0-\hat{t},t_0)$ and
\begin{align}\label{mathb}
\mathfrak{b}:=\max\left\{\|\nabla c\|_{L^{ \infty}(\Omega\times(t_0-\hat{t},t_0))  }\, , \,1\right\}<\infty,
\end{align}
then we have the following upper bound of $n$:
\begin{align}\label{th21}
\underset{\Omega\times(t_0-\hat{t}/2,t_0)}
{\rm ess~ sup} n
\leq C\mathfrak{b}^{\kappa  }
\left(\hat{t}+\frac{1}{  \hat{t}^{\frac{N}{p}}  }\right)^{ \hat{\kappa}  }
\left(\dashint_{t_0-\hat{t}}^{t_0}
\int_{ \Omega } n^{r} d x d t\right)^{   \hat{\kappa}         }+K_f,
\end{align}
where
$$r>\max\left\{\frac{  p(\beta+\alpha_-)  }{p-1}\, , \,2+\alpha_-\, , \, \frac{N(2+\alpha_-)}{p}
-N\, , \,  \frac{(p+N)(\beta+\alpha_-)}{p(p-1)}-N-2-\alpha_- \right\}, $$
and $K_f$ is given in \eqref{Kf}, the positive constants $\kappa$, $\hat{\kappa}$ and $C$ depend on \verb"data".
\end{theorem}
\begin{remark}\label{r3}
One may say the Lipschitz continuity assumption on $c$ required in \eqref{tmathb} and \eqref{mathb} is strong.
Actually, the regularity of $c$, as an interconnected component to $n$ within a system, can be obtained through utilizing the semigroup estimates and Sobolev estimates for parabolic equation \eqref{q1}$_2$. Here, we give a sufficient condition ensuring \eqref{tmathb} without proof:
assume that
$$
g(x,t)=g(n(x,t),x,t)\in L^{q}(Q_T)
$$
with some $q>N+2$, then for any $t_0\in(0,T)$ we have $\|\nabla c\|_{L^{ \infty}(\Omega\times(t_0,T))  }<\infty$.
\end{remark}

 When considering the H\"{o}lder-continuity of solutions, 
 it becomes necessary to additionally impose the following conditions on $a$ and $b$:
\begin{align}\label{struc}
a(\xi, s,x,t)=\bar{a}(s,x,t)|\xi|^{p-2},
	 \qquad\xi\in\Bbb{R}^N,~ s \geq 0,~x\in\overline{\Omega},~t>0,
\end{align}
where $\bar{a}$ as well as $b$ satisfies that for any $(s,x,t)\in\Bbb{R}^+\times\overline{\Omega}\times\Bbb{R}^+ $ 
and $(s_1,x_1,t_1)\in\Bbb{R}^+\times\overline{\Omega}\times\Bbb{R}^+ $,
\begin{align}\label{Hola}
|\bar{a}(s,x,t)-\bar{a}(s_1,x_1,t_1)|\le C_H\big( |s-s_1|+|x-x_1|+|t-t_1|\big)^{\omega_a}
\end{align}
and  
\begin{align}\label{Holb}
|b(s,x,t)-b(s_1,x_1,t_1)|\le \hat{C}_H\big( |s-s_1|+|x-x_1|+|t-t_1|\big)^{\omega_b}
\end{align}
with some $C_H,\hat{C}_H>0$ and $\omega_a,\omega_b\in(0,1)$.

\begin{theorem} [\label{th3}H\"{o}lder continuity]
Assume that $T\in(0,\infty]$ and $\Omega\subseteq\mathbb{R}^N$ $(N\ge 2)$ is a bounded domain with smooth boundary.
Let a pair of functions $(n,c)$ with $n\ge0$ in $Q_T$ be a weak solution to \eqref{q1} in the sense of Definition \ref{def},
where \eqref{Sa}--\eqref{Sf}, \eqref{mu} and \eqref{struc}--\eqref{Holb} are in force. For any $t_0\in(0,T)$ and every $\hat{T}\in(t_0,T)$, if $n$ is bounded in $\overline{\Omega}\times[t_0,\hat{T}]$, then there holds that $n\in C^{\gamma,\frac{\gamma}{p}}\big(\overline{\Omega}\times[t_0+\varepsilon,\hat{T}]\big)$ with each $ \varepsilon\in(0,\hat{T}-t_0)$ and some $\gamma\in(0,1)$. Furthermore, we also have the H\"{o}lder continuity of gradient as follows:
\begin{itemize}
\item [{\rm (i)}] If $p>2$, then we can find $\omega\in(0,1)$ and $C>1$ depending upon
$\varepsilon, p, \alpha, \beta, a_0, b_0, C_H, \hat{C}_H, \omega_a, \omega_b, N, \Omega$, $\|n\|_{L^{\infty}(\Omega\times(t_0, \hat{T}))  }$, $\|n\|_{W^{1,p}(\Omega\times(t_0,\hat{T}))  }$ and $\|\mathbf{u}\|_{L^{\infty}(\Omega\times(t_0, \hat{T}))  }$  such that for any $V\subset\subset \Omega$,
\begin{align*}
	|\nabla n(x, t)  -\nabla n(\hat{x}, \hat{t})|
	\leq C\left(\frac{ |x-\hat{x}|+\sqrt{|t-\hat{t}|} }{\operatorname{dist}(V ; \partial\Omega)}\right)^{\omega},
	\qquad\forall~ (x,t),~(\hat{x},\hat{t})\in \overline{V}\times[t_0+\varepsilon,\hat{T}].
\end{align*}
\item [{\rm (ii)}]  Suppose $p=2$. Then there exist $\omega\in(0,1)$ and $C>1$ depending upon $\varepsilon,p,\alpha,\beta,a_0,b_0,C_H,\hat{C}_H,\omega_a$, $\omega_b,N,\Omega$, $\|n\|_{L^{ \infty}(\Omega\times(t_0,\hat{T}))  }$ and $\|\mathbf{u}\|_{L^{\infty}(\Omega\times(t_0, \hat{T}))  }$ such that
\begin{align*}
	|\nabla n(x, t)-\nabla n(\hat{x}, \hat{t})|
	\leq C\left(|x-\hat{x}|+\sqrt{|t-\hat{t}|}\right)^{\omega},\qquad\forall~ (x,t),~(\hat{x},\hat{t})\in \overline{\Omega}\times[t_0+\varepsilon,\hat{T}].
\end{align*}
\end{itemize}
\end{theorem}

As evidenced by findings in, for instance, \cite{W2010, W2013}, chemotaxis models can exhibit an infinite-time explosion phenomenon under certain conditions that are satisfied by parameters in systems and initial data. Therefore, if the $L^\infty$-norm of the solution remains bounded as the time-variable $t\rightarrow\infty$, then the infinite-time explosion can be ruled out in the systems. In addition, for models involving volume filling effects, determining asymptotic stability also necessitates a well-suited upper bound estimate of the solution beforehand. Hence, in the upcoming theorem, we will conduct a quantitative analysis of such bound of the global solutions for further use.

\begin{theorem} [\label{th4}Boundedness result]
Let $\Omega\subseteq\Bbb{R}^N$ $(N\ge 2)$ be a bounded domain with a smooth boundary.
Assume that $T=\infty$, and \eqref{Sa}--\eqref{Sf} as well as \eqref{mu} are in force. 
Let $m>N$ and
$$r>\max\bigg\{  \frac{    p(\beta+\alpha_-)N  }{p-1}\, , \,2+\alpha_-\, , \, \frac{N(2+\alpha_-)}{p}
-N\, , \,  \frac{(p+N)(\beta+\alpha_-)}{p(p-1)}-N-2-\alpha_-\, , \,\frac{p(N+2+\alpha_-)}{N} \bigg\}.$$
be arbitrarily chosen constants.
Then we can obtain pure constants $\Lambda>0$ and
$\kappa>0$, determined a priori only in terms of \verb"data", such that
for any globally bounded solution $(n,c)$ with $n\ge0$ in $Q_T$ to \eqref{q1} in the sense of Definition \ref{def}, 
if there exist $K>0$ and a time point $\bar{t}_0>0$ ensuring
$$\|\mathbf{u}(\cdot,t)\|_{L^\infty(\Omega)}<K,
\quad\|n(\cdot,t)\|_{L^{r}(\Omega)}\le K,\qquad \forall~t>\bar{t}_0$$
and
 $$\|g(\cdot,t)\|_{L^{m}(\Omega)}\le K,\qquad \forall~t>\bar{t}_0,$$
 where $$g(x,t):=g(n(x,t),x,t),\qquad x\in\Omega,~t>0,$$
then we can find $\bar{T}>\bar{t}_0$ satisfying
  \begin{align}\label{m0}
	\|n(\cdot,t)\|_{L^\infty(\Omega)} \le \Lambda(K+1)^\kappa+K_f,\qquad \forall~t>\bar{T},
  \end{align}
and $K_f$ is given in \eqref{Kf}.
\end{theorem}

\noindent
{\bf Structure of the present work.}
Section \ref{sec2} in this work is primarily devoted to the proof of Theorems \ref{th1}\&\ref{th2} using the De Giorgi iteration method, 
with a focus on deriving local upper bound estimates for weak solutions.
On the basis of Theorems \ref{th1}\&\ref{th2}, we proceed to present the H\"{o}lder continuity of component $n$ through the proof of Theorem \ref{th3} in the remaining part of Section  \ref{sec2}, 
and then utilize Section  \ref{sec3} to build an initial-independent $L^\infty$-bound of $n$ as exhibited in Theorem \ref{th4}.
 The former can help us to improve the regularity of solutions to models involving the $p$-Laplace operator,
 while the latter can be used to determine asymptotic stability of solutions.
 
To better illustrate the role of our results in obtaining asymptotic stability and enhancing regularity for chemotaxis systems,
Section  \ref{sec4}  is devoted to providing four examples that discuss the long-time behaviors or regularity of solutions, respectively.
To highlight the broad applicability of our conclusions, we have selected models from different categories in these examples,
which concern chemotaxis models with nonlinear signal production, quasilinear chemotaxis-Navier-Stokes models,
$p$-Laplacian chemotaxis-consumption models, and chemotaxis-hypotaxis models.
Certainly, the applications of our theorems obtained here are not limited to the aforementioned models,
but due to space constraints, only a few examples are provided in detail here.

\section{Boundedness and H\"{o}lder continuity}
\label{sec2}
\begin{lemma} [\label{lem2.4}Caccioppoli inequality]
Let $(n,c)$ with $n\ge0$ in $Q_T$ be a weak solution to \eqref{q1}
under the conditions \eqref{Sa}-\eqref{Sf}.
Then we can find a pure constant
$C=C(\verb"data")>0$ such that 
for any cut-off function $\eta\in C_0^\infty((0,T))$, if there holds that 
$\nabla c\in L^\infty(\Omega\times {\rm supp} \eta)$,
then the following
inequality 
\begin{align}\label{d0}
&\quad\sup_{t\in(0,T)}\int_{\Omega}\big(n(x,t)-k\big)_{+}^{2+\alpha_-}\eta(t)  \,dx
+ \int^{T}_{0} \int_{ \Omega }
\left|\nabla(n-k)_{+}\right|^{p}(n+1)^{-\alpha_-}(n-k)_+^{\alpha_-}\eta \,dxd\tau\nonumber\\
&\le 
C\int^{T}_{0} \int_{A^+(k,\tau)}(n-k)_{+}^{2+\alpha_-}\eta_t(\tau)\,dxd\tau
+C\|\nabla c\|^{  \frac{p}{p-1}  }_{L^{ \infty  }(\Omega\times{\rm supp}\, \eta)  }
\int^{T}_{0}\int_{A^+(k,\tau)  } (n+1)^{ \frac{p(\beta+\alpha_-)}{p-1}  }\eta \,dxd\tau\nonumber\\
&\quad+C\int^{T}_{0} \int_{\Omega}f(n,x,\tau) (n-k)^{1+\alpha_-}_{+}\eta \,dxd\tau
\end{align}
is valid for every level $k>1$.
\end{lemma}
\begin{proof}[\bf Proof.]
{\bf Case $\alpha\ge0$.}
 In the weak formulation \eqref{De1} take the testing functions,
 $$\varphi(\cdot,t)=\eta(t) (n_h(\cdot,t)-k)_+, \qquad t\in(0,T),$$
 where $n_h$ is the Steklov average of $n$ given by
\begin{equation}
n_h (\cdot,t)\equiv
\begin{cases}
\frac{1}{h} \int_t^{t+h}n(\cdot, \tau) d \tau, & t \in(0, T-h], \\ 0, & t>T-h.
\end{cases}
\end{equation}
Through the arguments as proceeded in \cite[Proposition 3.1, Chapter II]{D1993}, we let $h\rightarrow0$ and hence derive that for any $t\in(0,T)$,
\begin{align}\label{d2}
&\quad\sup_{t\in(0,T)}\int_{\Omega}\big(n(x,t)-k\big)_{+}^{2}\eta(t) \,dx
+\int^{T}_{0} \int_{\Omega}a( \nabla n ,n, x, \tau )\nabla n\cdot \nabla(n-k)_{+}\eta(\tau) \,dxd\tau\nonumber\\
&\le \int^{T}_{0} \int_{\Omega}(n-k)_{+}^{2}\eta_t(\tau)\,dxd\tau
+\int^{T}_{0}\int_{\Omega } b(n, x, \tau )\nabla c\cdot\nabla\left(n-k\right)_{+}\eta \,dxd\tau\nonumber\\
&\quad+\int^{T}_{0} \int_{\Omega }f(n,x,\tau) (n-k)_{+}\eta \,dxd\tau,
\end{align}
where the second integral on the left-hand side can be estimated by (\ref{Sa}) as below,
\begin{align}\label{d3}
\int^{T}_{0}\int_{\Omega}a( \nabla n ,n, x, \tau )\nabla n\cdot \nabla(n-k)_{+}\eta\,dxd\tau
&\ge a_{0}\int^{T}_{0} \int_{ \Omega }
\left|\nabla(n-k)_{+}\right|^{p}(n+1)^{\alpha}\eta\,dxd\tau\nonumber\\
&\ge a_0\int^{T}_{0} \int_{ \Omega }
\left|\nabla(n-k)_{+} \right|^{p}\eta \,dxd\tau.
\end{align}
It follows by using \eqref{Sb} and Young's inequality that
\begin{align}\label{d4}
&\quad\int^{T}_{0} \int_{\Omega} b( n,x, \tau)\nabla\left(n-k\right)_{+}\cdot\nabla c\eta \,dxd\tau\nonumber\\
&\le b_{0}
\int^{T}_{0} \int_{A^+(k,\tau)} (n+1)^\beta\left| \nabla\left(n-k\right)_{+}\right|
\left| \nabla c\right| \eta \,dxd\tau\nonumber\\
&\leq \frac{a_{0}}{4} \int^{T}_{0}\int_{A^+(k,\tau)}
\left|\nabla\left(n-k\right)_{+}\right|^{p} \eta \,dxd\tau\nonumber\\
&\quad+\Big(\frac{4b_0^p}{a_0}\Big)^{\frac{1}{p-1}}\|\nabla c\|^{  \frac{p}{p-1}  }_{L^{ \infty  }(\Omega\times{\rm supp}\, \eta)  }
\int^{T}_{0}\int_{A^+(k,\tau)  } (n+1)^{ \frac{p\beta}{p-1} }\eta \,dxd\tau.
\end{align}
A combination of \eqref{d2}--\eqref{d4} guarantees our claim.\\
{\bf Case $\alpha<0$.}
Utilizing reasoning as leading to \eqref{d2}, we have via the test function $\eta(t)(n_h(\cdot,t)-k)^{1-\alpha}_+$ that
\begin{align}\label{d2'}
&\quad\sup_{t\in(0,T)}\frac{1}{2-\alpha}\int_{\Omega}\big(n(x,t)-k\big)_{+}^{2-\alpha}\eta(t) \,dx
+\int^{T}_{0} \int_{\Omega}a( \nabla n,  n, x, \tau)\nabla n\cdot \nabla(n-k)_{+}^{1-\alpha}\eta(\tau) \,dxd\tau\nonumber\\
&\le\frac{1}{2-\alpha}\int^{T}_{0} \int_{\Omega}(n-k)_{+}^{2-\alpha}\eta_t(\tau)\,dxd\tau
+\int^{T}_{0}\int_{\Omega} b( n,x, \tau)
\nabla c\cdot\nabla\left(n-k\right)_{+}^{1-\alpha}\eta(\tau) \,dxd\tau\nonumber\\
&\quad
+\int^{T}_{0} \int_{\Omega}f(n,x,\tau) (n-k)^{1-\alpha}_{+}\eta(\tau)\,dxd\tau.
\end{align}
The conditions \eqref{Sa} and \eqref{Sb} imply that
\begin{align}\label{d3'}
&\quad\int^{T}_{0}\int_{\Omega }
a( \nabla n,  n, x, \tau)
\nabla n\cdot \nabla(n-k)_{+}^{1-\alpha}\eta \,dxd\tau\nonumber\\
&\ge a_{0}(1-\alpha)\int^{T}_{0} \int_{\Omega }
\left|\nabla(n-k)_{+}\right|^{p}(n+1)^{\alpha}(n-k)_+^{-\alpha}\eta \,dxd\tau
\end{align}
and
\begin{align}\label{d4'}
&\quad\int^{T}_{0} \int_{\Omega } b( n ,x, \tau)\nabla\left(n-k\right)_{+}^{1-\alpha}\cdot\nabla c\eta \,dxd\tau\nonumber\\
&\le b_{0}(1-\alpha)
\int^{T}_{0} \int_{A^+(k,\tau)} (n-k)^{-\alpha}_+(n+1)^\beta\left|\nabla \left(n-k\right)_{+}\right| |\nabla c|\eta \,dxd\tau\nonumber\\
&\leq \frac{a_{0}(1-\alpha) }{4} \int^{T}_{0}\int_{A^+(k,\tau)}
\left|\nabla \left(n-k\right)_{+}\right|^{p}(n+1)^{\alpha}(n-k)^{-\alpha}_+\eta \,dxd\tau\nonumber\\
&\quad+
\Big(\frac{4b_0^p }{a_0}\Big)^{\frac{1}{p-1}}(1-\alpha)
\|\nabla c\|^{  \frac{p}{p-1}  }_{L^{ \infty  }(\Omega\times{\rm supp}\, \eta)  }
\int^{T}_{0}\int_{A^+(k,\tau)  } (n+1)^{ \frac{p(\beta-\alpha)}{p-1}  }\eta \,dxd\tau,
\end{align}
where in the last line we utilized Young's inequality. In light of \eqref{d2'}, \eqref{d3'} and \eqref{d4'}, we conclude that
\begin{align}\label{d5'}
&\quad\frac{1}{2-\alpha}\sup_{t\in(0,T)}\int_{\Omega}\big(n(x,t)-k\big)_{+}^{2-\alpha}\eta(t) \,dx\nonumber\\
&\quad+ \frac{a_0(1-\alpha)}{2}\int^{T}_{0} \int_{ \Omega }
\left|\nabla(n-k)_{+}\right|^{p}(n+1)^{\alpha}(n-k)_+^{-\alpha}\eta \,dxd\tau\nonumber\\
&\le (1-\alpha)\Big(\frac{4b_0^p}{a_0}\Big)^{\frac{1}{p-1}}\|\nabla c\|^{  \frac{p}{p-1}  }_{L^{ \infty  }(\Omega\times{\rm supp}\, \eta)  }
\int^{T}_{0}\int_{A^+(k,\tau)  } (n+1)^{ \frac{p}{p-1}(\beta-\alpha)  }\eta \,dxd\tau
\nonumber\\
&\quad+\int^{T}_{0} \int_{\Omega}f(n,x,\tau) (n-k)^{1-\alpha}_{+}\eta \,dxd\tau
+\frac{1}{2-\alpha}\int^{T}_{0} \int_{A^+(k,\tau)}(n-k)_{+}^{2}\eta_t(\tau)\,dxd\tau,
\end{align}
as expected.
\end{proof}

 With recalling the definition of $K_f$ in \eqref{Kf},
we select $k_0\geq K_f$ and choose sequences of increasing levels as
\begin{align}
&k_{j}:=\left(2-2^{-j}\right) k_0, \quad \tilde{k}_{j}:=\frac{k_{j+1}+k_{j}}{2},
 ~~~ j=0,1,2,\ldots.
\end{align}
Assume that $t_0\in(0,T)$ and $\hat{t}\in (0,T)$ are given as in Theorem \ref{th1} or Theorem \ref{th2}, 
and let $\hat{\tau}\in(0,\hat{t})$ and $\sigma\in(0,1)$ be specified later.
We set the intervals as
\begin{align}\label{Gam}
\Gamma_{j}:=\big(t_0-\sigma \hat{\tau}-2^{-j}(1-\sigma)\hat{\tau},t_0\big),
\qquad j=0, 1, 2, \cdots.
\end{align}
The relation $\Gamma_{j+1}\subseteq\Gamma_{j}$ clearly holds for all $j\in\mathbb{N}$.
Let
$\eta_j\in C_0^\infty((0,T))$ be the cut-off function w.r.t the time interval $\Gamma_{j}$ and satisfy
$${\rm supp}~ \eta_j\subseteq \Gamma_{j},
\quad | \eta_{j}'|\le\frac{C2^j}{(1-\sigma)\hat{\tau}}\,\,{\rm in}\,\,\Gamma_{j}
,\quad 0\le \eta_j\le 1\,\,{\rm in}\,\,\Gamma_{j}\,\,\,
{\rm and}\,\,\,\eta_j=1\,\,{\rm in}\,\,\Gamma_{j+1}.$$

\begin{lemma}\label{lem2.5}
Let $(n,c)$ with $n\ge0$ in $Q_T$ be a weak solution of \eqref{q1} under the conditions \eqref{Sa}--\eqref{Sf}.
Let $r$ be a number satisfying
$$r>\max\left\{ \frac{p(\beta+\alpha_-)}{p-1} \,,\,2+\alpha_-\,,\,p \right\}.$$
Assume that $n\in L^{r}(\Omega\times(t_0-\hat{t},t_0))$
and $\nabla c\in L^{\infty}(\Omega\times(t_0-\hat{t},t_0))$.
 Then
 we can find a pure constant $C=C(\verb"data")>0$ such that for every level $k_j\ge K_f$, the following
inequality is valid
\begin{align}\label{m0}
&\quad\sup _{t\in\Gamma_{j+1}}
\int_{\Omega}w_{j+1}^{2+\alpha_-}(\cdot, \tau)\,dx
+\int_{ \Gamma_{j+1} }\int_{\Omega} |\nabla w_{j+1} |^{p} \,dxd\tau
+\int_{ \Gamma_{j+1} }\int_{\Omega}  w_{j+1} ^{p} \,dxd\tau\nonumber\\
& \leq
\frac{  C 4^{  (r+\alpha_-)j   } 
\|\nabla c\|^{  \frac{p}{p-1}  }_{L^{ \infty}(\Omega\times(t_0-\hat{t},t_0))  } }{  k_0^{ r-\frac{p (\beta+\alpha_-)}{p-1}    }  }
\int_{\Gamma_j}\int_{ \Omega }w_{j-1}^{  r  } \,dxd\tau
+\frac{  C 2^{(r+\alpha_-)j} }{(1-\sigma)\hat{\tau} k_0^{ r-2-\alpha_-  } }
\int_{\Gamma_j} \int_{ \Omega } w_{j-1} ^{ r }\,dxd\tau\nonumber\\
&
\quad+ \frac{  C 2^{(r+\alpha_-)j} }{ k_0^{ r-p } } \int_{\Gamma_j}\int_{\Omega }w_{j-1}^{ r }\,dxd\tau,
\end{align}
where $\Gamma_j$ is defined in \eqref{Gam} and $w_j=(n-k_j)_+$.
\end{lemma}
\begin{proof}[\bf Proof.]
{\bf Case $\alpha\ge0$.}
It is not hard to verify that for all $\ell\ge m\ge0$ and $j\in \mathbb{N}^+$,
\begin{align}\label{em2}
w_{j-1}^\ell & \geq\left(n-k_{j-1}\right)_{+}^\ell \chi_{\left\{n \geq k_{j}\right\}}(x, t) \nonumber\\
& \geq\left(k_{j}-k_{j-1}\right)^{\ell-m}\left(n-k_{j-1}\right)_{+}^m \chi_{\left\{n \geq k_{j}\right\}}(x, t) \nonumber\\
& \geq  k_0^{\ell-m} 2^{-(\ell-m) (j+1)}\left(n-k_{j-1}\right)_{+}^m\chi_{\left\{n \geq k_{j}\right\}}(x, t)\\ \label{em1}
& \geq  k_0^{\ell-m} 2^{-(\ell-m) (j+1)} w_{j}^m \chi_{\left\{n \geq k_{j}\right\}}(x, t)\quad \text {in } Q_T.
\end{align}
By invoking Lemma \ref{lem2.4} and noticing \eqref{Kf} and $k_j\ge K_f$, we have
\begin{align}\label{g1}
&\quad\sup _{\tau\in\Gamma_j}\int_{\Omega}w_j^{2}\eta_j(\tau)\,dx
+ a_0\int_{\Gamma_j} \int_{A^+(k_j,\tau)}
\left|\nabla w_j \right|^{p}\eta_j \,dxd\tau\nonumber\\
&\le \int_{\Gamma_j} \int_{A^+(k_j,\tau)}w_j ^{2}|\eta'_{j}(\tau)|\,dxd\tau
+C\|\nabla c\|^{  \frac{p}{p-1}  }_{L^{\infty}(\Omega\times \Gamma_0)  }
\int_{\Gamma_j}\int_{A^+(k_j,\tau)  } (n+1)^{ \frac{p\beta}{p-1}   }\eta_j\,dxd\tau\nonumber\\
&\le \int_{\Gamma_j} \int_{A^+(k_j,\tau)} w_j^{2}|\eta'_{j}(\tau)|\,dxd\tau
+C\|\nabla c\|^{  \frac{p}{p-1}  }_{L^{  \infty }(\Omega\times \Gamma_0)  }
\int_{\Gamma_j}\int_{A^+(k_j,\tau)  } n^{ \frac{p\beta}{p-1}   }\eta_j \,dxd\tau\nonumber\\
&
\quad+C\|\nabla c\|^{  \frac{p}{p-1}  }_{L^{ \infty }(\Omega\times \Gamma_0)  }
\int_{\Gamma_j}|A^+(k_j,\tau)|\,d\tau ,
\end{align}
where the first term on the right-hand side can be estimated by \eqref{em1} as below,
\begin{align}\label{m1}
 \int_{\Gamma_j} \int_{A^+(k_j,\tau)}w_j ^{2}|\eta_{j}'(\tau)|\,dxd\tau
\leq \frac{ C 2^{r j } }{ (1-\sigma)\hat{\tau} k_0^{ r-2 } }
\int_{\Gamma_j} \int_{ \Omega } w_{j-1} ^{ r }\,dxd\tau.
\end{align}
Similarly, we also have
 \begin{align}\label{m11}
\int_{\Gamma_j} \int_{A^+(k_j,\tau)}w_j ^{p} \,dxd\tau
\leq \frac{2^{r (j+1)} }{ k_0^{ r-p }  }
\int_{\Gamma_j} \int_{ \Omega } w_{j-1} ^{ r }\,dxd\tau.
\end{align}
By virtue of \eqref{em2}, there holds that
\begin{align}\label{m12}
\int_{\Gamma_j}\int_{\Omega }w_{j-1}^{  r } \,dxd\tau
&\ge\frac{ k_0^{  r   } }{2^{ r(j+1) } }
\int_{\Gamma_j}\left|A^+(k_j,\tau)\right|\,d\tau,
\end{align}
and
\begin{align}\label{m41}
\int_{\Gamma_j}\int_{\Omega }w_{j-1}^{ r } \,dxd\tau
& \geq\frac{Ck_0^{ r-\frac{p\beta}{p-1}  } }{  2^{ (r-\frac{p\beta}{p-1}) j  }  }
\int_{\Gamma_j}\int_{ A^+(k_{j},\tau) }
w_{j-1}^{  \frac{p\beta}{p-1}  } \,dxd\tau  \nonumber\\
& \geq \frac{Ck_0^{ r-\frac{p\beta}{p-1}  } }{  2^{ (r-\frac{p\beta}{p-1}) j  }  }
\int_{\Gamma_j}\int_{  A^+(k_{j},\tau)  } n^{   \frac{p\beta}{p-1}   }
\left(1-\frac{k_{j-1}}{k_{j}}\right)^{  \frac{p\beta}{p-1}   }\,dxd\tau\nonumber\\
& \geq \frac{Ck_0^{  r-\frac{p\beta}{p-1}    }}{ 2^{   r(j+1)     } }
\int_{ \Gamma_j }\int_{A^+(k_{j},\tau)  }n^{   \frac{p\beta}{p-1}   }\,dxd\tau.
\end{align}
We substitute \eqref{m1}--\eqref{m41} into \eqref{g1} to derive that
\begin{align}\label{m5}
&\quad\sup _{t\in\Gamma_{j+1}}\int_{\Omega}w_{j}^{2}\,dx
+\int_{\Gamma_{j+1}}\int_{\Omega} |\nabla w_{j}|^{p} \,dxd\tau
+\int_{\Gamma_{j+1}} \int_{ \Omega }w_j ^{p} \,dxd\tau\nonumber\\
& \leq \frac{ C 2^{ rj  } \|\nabla c\|^{  \frac{p}{p-1}  }_{L^{ \infty}(\Omega\times(t_0,T) )  } }{  k_0^{ r-\frac{p \beta}{p-1}    }  }
\int_{\Gamma_j}\int_{ \Omega }w_{j-1}^{  r  } \,dxd\tau
+\frac{  C 2^{rj} }{(1-\sigma)\hat{\tau} k_0^{ r-2  } }
\int_{\Gamma_j} \int_{ \Omega } w_{j-1} ^{ r }\,dxd\tau\nonumber\\
&+ \frac{  C 2^{rj} }{ k_0^{ r-p } } \int_{\Gamma_j}\int_{\Omega }w_{j-1}^{ r } \,dxd\tau.
\end{align}
{\bf Case $\alpha<0$.}
By employing Lemma \ref{lem2.4} again, we have that
\begin{align}\label{d5}
&\quad\sup _{t\in \Gamma_j}\int_{\Omega}w_j^{2-\alpha}\eta_j(\tau) \,dx
+ \int_{\Gamma_j} \int_{\Omega}
\left|\nabla w_j\right|^{p}(n+1)^{\alpha}w_j^{-\alpha}\eta_j \,dxd\tau\nonumber\\
&\le \int_{\Gamma_j} \int_{A^+(k_j,\tau)} w_j^{2-\alpha}|\eta'_{j}(\tau)|\,dxd\tau
+C\|\nabla c\|^{  \frac{p}{p-1}  }_{L^{  \infty  }(\Omega\times\Gamma_0 )  }
 \int_{\Gamma_j}\int_{A^+(k_j,\tau)  } (n+1)^{ \frac{p(\beta+\alpha_-)}{p-1}  }\eta_j \,dxd\tau\nonumber\\
&\le \int_{\Gamma_j} \int_{A^+(k_j,\tau)} w_j^{2-\alpha}|\eta'_{j}(\tau)|\,dxd\tau
+C\|\nabla c\|^{  \frac{p}{p-1}  }_{L^{ \infty}(\Omega\times\Gamma_0  )  }
\int_{\Gamma_j}\int_{A^+(k_j,\tau)  } n^{ \frac{p(\beta+\alpha_-)}{p-1}   }\eta_j \,dxd\tau\nonumber\\
&\quad+C\|\nabla c\|^{  \frac{p}{p-1}  }_{L^{  \infty  }(\Omega\times\Gamma_0 )  }
\int_{\Gamma_j}|A^+(k_j,\tau)|\,d\tau.
\end{align}
Noticing the monotonicity of the function $f_j(x):=(x-k_j)^{-\alpha}(x+1)^\alpha$ defined on $(k_j,\infty)$,
we obtain that
\begin{align}\label{d5'}
\int_{\Gamma_j} \int_{\Omega}
\left|\nabla w_j\right|^{p}(n+1)^{\alpha} w_j^{-\alpha}\eta_j \,dxd\tau
\ge &\int_{\Gamma_j} \int_{ A^+(k_{j+1},\tau) } \left|\nabla w_j\right|^{p}(k_{j+1}+1)^{\alpha}(k_{j+1}-k_j)_+^{-\alpha}\eta_j \,dxd\tau\nonumber\\
=& \frac{k_0^{-\alpha} }{ (k_0+2^{j+1})^{-\alpha} }
\int_{\Gamma_j} \int_{ A^+(k_{j+1},\tau) } \left|\nabla w_j\right|^{p}\eta_j \,dxd\tau\nonumber\\
\ge &2^{(j+2)\alpha}\int_{\Gamma_j} \int_{ A^+(k_{j+1},\tau) } \left|\nabla w_j\right|^{p}\eta_j  \,dxd\tau\nonumber\\
\ge &2^{(j+2)\alpha}\int_{\Gamma_{j+1}} \int_{ \Omega } \left|\nabla w_{j+1}\right|^{p} \,dxd\tau.
\end{align}
By utilizing the similar arguments as leading to \eqref{m1} and \eqref{m41}, we obtain
\begin{align}\label{m1'}
 \int_{\Gamma_j} \int_{A^+(k_j,\tau)}w_j ^{2-\alpha}|\eta_{j}'(\tau)|\,dxd\tau
\leq \frac{  C2^{r j } }{ (1-\sigma)\hat{\tau} k_0^{ r-2+\alpha } }
\int_{\Gamma_j} \int_{ \Omega } w_{j-1} ^{ r }\,dxd\tau
\end{align}
and
\begin{align}\label{m41'}
\int_{ \Gamma_j }\int_{ A^+(k_j,\tau) }
n^{ \frac{p(\beta+\alpha_-)}{p-1}   }\,dxd\tau
\leq\frac{  2^{   r(j+1)     } }{k_0^{  r-\frac{p(\beta-\alpha)}{p-1}    }}
\int_{\Gamma_j}\int_{\Omega }w_{j-1}^{ r } \,dxd\tau.
\end{align}
A combination of \eqref{d5}--\eqref{m41'}, \eqref{m11} and \eqref{m12} implies that
\begin{align}\label{m5'}
&\quad\sup _{t\in\Gamma_{j+1}}
\int_{\Omega}w_{j+1}^{2-\alpha}(\cdot, \tau)\,dx
+\int_{ \Gamma_{j+1} }\int_{\Omega} |\nabla w_{j+1} |^{p} \,dxd\tau
+\int_{ \Gamma_{j+1} }\int_{\Omega}  w_{j+1} ^{p} \,dxd\tau\nonumber\\
& \leq
\frac{ C 2^{ (r-\alpha)j} \|\nabla c\|^{  \frac{p}{p-1}  }_{L^{ \infty}(\Omega\times(t_0,T))  } }{  k_0^{ r-\frac{p (\beta+\alpha_-)}{p-1}}}
\int_{\Gamma_j}\int_{ \Omega }w_{j-1}^{  r  } \,dxd\tau
+\frac{  C 2^{(r-\alpha)j} }{(1-\sigma)\hat{\tau} k_0^{ r-2-\alpha_-  } }
\int_{\Gamma_j} \int_{ \Omega } w_{j-1} ^{ r }\,dxd\tau\nonumber\\
&\quad+\frac{  C 2^{  (r-\alpha)j     } }{k_0^{  r-p    }}
\int_{\Gamma_j}\int_{\Omega }w_{j-1}^{ r } \,dxd\tau.
\end{align}
This together with \eqref{m5} allows us to arrive at the claim immediately.
\end{proof}

The following inequality can be obtained by an application of the H\"{o}lder inequality and Sobolev embedding inequality, as proved in \cite[Chapter I, Proposition 3.1]{D1993} and \cite[Lemma 2.1]{L1987}.
\begin{lemma}\label{lem2.6}
Let $T>0$, $m\ge 1$, $p\ge 1 $ and $\Omega\subseteq\Bbb{R}^N$ $(N\ge2)$ be a bounded domain with Lipschitz boundary. There exists a constant $C$ depending only upon $N, p, m$
 and the structure of $\partial\Omega$ such that any $\varphi \in L^{\infty}\big(0, T ; L^{m}(\Omega)\big) \cap
L^{p}\big(0, T ; W^{1, p}(\Omega)\big)$ satisfies
\begin{align*}
\int_{Q_{T}}|\varphi(x, t)|^{\frac{p(N+m)}{N}} \,dxdt
\leq
C\left(\int_{Q_{T}}|\nabla \varphi(x, t)|^{p}+| \varphi(x, t)|^{p} \,dxdt\right)
\left(\underset{0<t<T}
{\rm ess~ sup}\int_{\Omega}|\varphi(x, t)|^{m} \,dx\right)^{\frac{p}{N}}.
\end{align*}
\end{lemma}
\begin{lemma} [Lemma 4.1, \cite{D1993}\label{lem-iteration}]
Let $\left\{Y_{j}\right\}_{j\in \Bbb{N}}$ be a sequence of positive
numbers, satisfying the recursive inequalities
\begin{align*}
Y_{j+1} \leq K b^{j} Y_{j}^{1+\delta}, ~~ j=0,1,2,\ldots,
\end{align*}
where $K, b>1$ and $\delta>0$ are given numbers. If
$$
Y_{0} \leq K^{ -\frac{1}{\delta} } b^{  -\frac{1}{\delta^2 }   },
$$
then $Y_{j}$ converges to zero as $j\rightarrow\infty$.
\end{lemma}
\begin{proof}
[\bf Proof of Theorem \ref{th1}.]
Let the assumptions \eqref{Sa}--\eqref{Sf} be satisfied as well as \eqref{tmathb} and
$(n,c)$ with $n\ge0$ in $Q_T$ be a weak solution to the model \eqref{q1} in the sense of Definition \ref{def}.
Set $$r:=\frac{p(N+2+\alpha_-)}{N}.$$
By \eqref{thm1-1}, we further define $$\hat{\theta}:=r -\max\left\{\frac{p(\beta+\alpha_-)}{p-1}\,,\, 2+\alpha_-\,,\, p\right\}>0.$$
Invoking Lemma \ref{lem2.5} with taking $\sigma=1/2$ and $\hat{\tau}=\hat{t}$, and recalling the definition of $\mathfrak{b}$, one sees that
\begin{align}\label{pt3}
&\quad\sup _{t\in\Gamma_{j+1}}
\int_{\Omega}w_{j+1}^{2+\alpha_-}(\cdot, t)\,dx
+\int_{\Gamma_{j+1} }\int_{\Omega} |\nabla w_{j+1}|^{p} \,dxd\tau
+\int_{\Gamma_{j+1} }\int_{\Omega}  w_{j+1}^{p} \,dxd\tau\nonumber\\
& \leq \frac{ C2^{  rj   } \mathfrak{b}^{  \frac{p}{p-1}  }      }{  k_0^{ r-\frac{p (\beta+\alpha_-)}{p-1}    }  }
\int_{\Gamma_j}\int_{ \Omega }w_{j-1}^{  r  } \,dxd\tau
+\frac{  C 2^{rj} }{\hat{t} k_0^{ r-2-\alpha_-  } }
\int_{\Gamma_j} \int_{ \Omega } w_{j-1} ^{ r }\,dxd\tau
+\frac{  C 2^{rj} }{ k_0^{ r-p  } }
\int_{\Gamma_j} \int_{ \Omega } w_{j-1} ^{ r }\,dxd\tau\nonumber\\
& \leq C\Big(1+\frac{1}{\hat{t}} \Big)
\frac{ 2^{rj }\mathfrak{b}^{  \frac{p}{p-1}  }  }{  k_0^{  \hat{\theta}   }  }
\int_{\Gamma_j}\int_{ \Omega }w_{j-1}^{  r  } \,dxd\tau,
\end{align}
where the condition $k_0\ge1$ is employed in the last line. 
According to Lemma \ref{lem2.6} applied to the function $w_{j+1}$,
\begin{align*}
\dashint_{\Gamma_{j+1} }\int_{\Omega} w_{j+1}^{\frac{p(N+2+\alpha_-)}{N}} \,dxd\tau
&\leq
C\left(\dashint_{\Gamma_{j+1} }\int_{\Omega} |\nabla w_{j+1} |^{p}+w^p_{i+1} \,dxd\tau\right)
\left(\underset{t\in\Gamma_{j+1}}
{\rm ess~ sup} \int_{\Omega}w_{j+1}^{2+\alpha_-}\,dx\right)^{\frac{p}{N}},
\end{align*}
which combining with \eqref{pt3} yields that
\begin{align}\label{pt4}
\dashint_{\Gamma_{j+1} }\int_{\Omega} w_{j+1}^{\frac{p(N+2+\alpha_-)}{N}} \,dxd\tau
 \leq C \Big(1+\frac{1}{\hat{t}} \Big)^{ \frac{p+N}{ N}  } \hat{t}^{\frac{p}{ N} }
 \frac{  2^{ \frac{jr(p+N)}{N}}\mathfrak{b}^{ \frac{p+N}{ N}\cdot\frac{p}{p-1}  }   }{  k_0^{\frac{(p+N)\hat{\theta}}{N}}  }
\left(\dashint_{\Gamma_{j-1}}\int_{ \Omega }w_{j-1}^{   r  } \,dxd\tau\right)^{\frac{p+N}{ N}}.
\end{align}
Now, let us define the quantity $Y_j$ as below,
$$Y_j:=\dashint_{\Gamma_{2j}}\int_{ \Omega }w_{2j}^{   \frac{p(N+2+\alpha_-)}{N} } \,dxd\tau
=\dashint_{\Gamma_{2j}}\int_{ \Omega }w_{2j}^{ r } \,dxd\tau,\qquad j=0,1,2
\ldots.$$
The estimate \eqref{pt4} directly gives a pure constant $\bar{ C}=\bar{C}(\verb"data")>0$ such that
\begin{align}
Y_j&\le \bar{C} \Big(1+\frac{1}{\hat{t}} \Big)^{ \frac{p+N}{ N}  } \hat{t}^{\frac{p}{ N} }
 \frac{  4^{ \frac{jr(p+N)}{N}}\mathfrak{b}^{ \frac{p(p+N)}{ N(p-1)}  }   }{  k_0^{\frac{(p+N)\hat{\theta}}{p}}  }
\left(\dashint_{\Gamma_{2j-2}}\int_{ \Omega }w_{2j-2}^{   r  } \,dxd\tau\right)^{\frac{p+N}{ N}}\nonumber\\
&\le \bar{ C} \Big(\hat{t}+\frac{1}{\hat{t}^{\frac{N}{p}}   } \Big)^{ \frac{p}{ N}  }
 \frac{   4^{ \frac{jr(p+N)}{N}}\mathfrak{b}^{  \frac{p(p+N)}{ N(p-1)}}   }{  k_0^{\frac{(p+N)\hat{\theta}}{N}}  }
Y_{j-1}^{  1+\frac{p}{ N}  }
,\qquad j=1,2,3\ldots.
\end{align}
If $Y_0$ satisfies the estimate
\begin{equation}
\begin{aligned}
Y_0=\dashint_{t_0-\hat{t}}^{t_0}\int_{\Omega }\left(n-k_{0}\right)_{+}^{     \frac{p(N+2+\alpha_-)}{N}        }\,dxd\tau
\leq (2\bar{ C})^{-\frac{N}{p}}  4^{-\frac{r(p+N)N}{p^2}} \Big(\hat{t}+\frac{1}{\hat{t}^{\frac{N}{p}}   } \Big)^{ -1 }
\mathfrak{b}^{   -\frac{N+p}{p-1}  } k_0^{\frac{\hat{\theta}(p+N)}{p}},
\end{aligned}
\end{equation}
then we can derive by invoking Lemma \ref{lem-iteration} that $Y_j
\rightarrow0$ as $j\rightarrow\infty$. Hence, we select
\begin{align*}
k_0\ge\max\Bigg\{&K_f \,,\,(2\bar{ C})^{\frac{N}{\hat{\theta}(p+N)}} 4^{\frac{Nr}{\hat{\theta}p}  }
 \mathfrak{b}^{ \frac{p}{ \hat{\theta} (p-1)  }  } \Big(\hat{t}+\frac{1}{  \hat{t}^{\frac{N}{p}}  } \Big)^{     \frac{p}{\hat{\theta}(p+N)}       }
 \left( \dashint_{t_0-\hat{t}}^{t_0}\int_{\Omega }n^{    \frac{p(N+2+\alpha_-)}{N}       }
\,dx d \tau\right)^{     \frac{p}{\hat{\theta}(p+N)}       }
\Bigg\}
\end{align*}
to ensure that
\begin{align*}
\underset{\Omega\times(t_0-\hat{t}/2,t_0)}
{\rm ess~ sup} n(\cdot, t)
\leq C  \mathfrak{b}^{ \frac{p}{ \hat{\theta} (p-1)  }  }  \Big(\hat{t}+\frac{1}{  \hat{t}^{\frac{N}{p}}  } \Big)^{     \frac{p}{\hat{\theta}(p+N)}       }
\left(\dashint_{t_0-\hat{t}}^{t_0}\int_{\Omega }n^{    \frac{p(N+2+\alpha_-)}{N}        }
\,dx d \tau\right)^{     \frac{p}{\hat{\theta}(p+N)}       }
+K_f,
\end{align*}
as desired.
\end{proof}

\medskip

\begin{proof}
[\bf Proof of Theorem \ref{th2}.]
Let the assumptions \eqref{A}--\eqref{Sf} as well as \eqref{mathb} be satisfied and
$(n,c)$ with $n\ge0$ in $Q_T$ be a weak solution of \eqref{q1}.
With defining $T_0=\hat{t}/2$, $T_i=\hat{t}/2+\sum^i_{k=1}2^{-k-1}\hat{t}$ ($i\in\Bbb{N}^+$) and
$$\hat{\Gamma}_{i}=(t_0-T_i,t_0),\qquad i=1,2,3,\cdots,$$
we choose $\hat{\tau}=T_{i+1}$ and $\sigma\hat{ \tau}=T_i$ in \eqref{Gam},
and thereby find a sequence of time interval $\{\Gamma_{j}\}_{j=0}^{+\infty}$
$$\Gamma_{j}:=\Gamma^{(i)}_{j}=\Big(t_0-T_i-2^{-j}(T_{i+1}-T_i),t_0\Big),\qquad j=0,1,2,\cdots.$$
Clearly, it can be checked that
$$1-\sigma=\frac{2^{-i-2}}{1/2+\sum^{i+1}_{k=1}2^{-k-1}}\ge2^{-i-2}$$
and
$$\hat{\Gamma}_{i}\subseteq\ldots\Gamma_{j+1}\subseteq \Gamma_{j}
\subseteq\ldots \subseteq\Gamma_{1}\subseteq\Gamma_{0}=\hat{\Gamma}_{i+1}\subseteq (t_0-\hat{t},t_0).$$
Then, we select some
$$
r>\max\left\{\frac{p(\beta+\alpha_-)}{p-1}\, , \,2+\alpha_-\, , \, \frac{N(2+\alpha_-)}{p}
-N\, , \,  \frac{(p+N)(\beta+\alpha_-)}{p(p-1)}-N-2-\alpha_- \right\}
$$
and still, let
$$\hat{\theta}:=r -\max\left\{\frac{p(\beta+\alpha_-)}{p-1}\, ,\,2+\alpha_-\right\}>0.$$
The choice of $r$ guarantees with \eqref{thm2-1} that
$$m:=\frac{  Nr- p(N+2+\alpha_-)  }{ \hat{\theta}(p+N)  }\in(0,1).$$
Now let us define the quantity $Y_j$ as below,
$$
Y_j:=\dashint_{\Gamma_{2j}}\int_{ \Omega }w_{2j}^{   r } \,dxd\tau,\quad j=0,1,2,\cdots
$$
and
$$M_i=\underset{\Omega\times\hat{\Gamma}_{i}}{\rm ess\sup }~n, \quad i=0,1,2,\cdots.$$
Thus, recalling the definition of $\mathfrak{b}$ and utilizing Lemmas \ref{lem2.5}--\ref{lem2.6}, there holds that
\begin{align}
Y_j&\leq \|w_{2j}\|^{   r- \frac{p(N+2+\alpha_-)}{N} }_{L^\infty(\Omega\times \Gamma_{2j})}
\dashint_{\Gamma_{2j}}\int_{ \Omega }w_{2j}^{   \frac{p(N+2+\alpha_-)}{N} } \,dxd\tau\nonumber\\
&\leq
C\|w_{2j}\|^{   r- \frac{p(N+2+\alpha_-)}{N} }_{L^\infty(\Omega\times\hat{\Gamma}_{i+1}     )}
\left(\dashint_{\Gamma_{2j} }\int_{\Omega} |\nabla w_{2j} |^{p}+w^p_{2j}\,dxd\tau\right)
\left(\sup _{t\in\Gamma_{2j}} \int_{\Omega}w_{2j}^{2+\alpha_-}\,dx\right)^{\frac{p}{N}}\nonumber\\
\nonumber\\
& \leq C\mathfrak{b}^{ \frac{p(p+N)}{ N(p-1)  }  } 2^{ \frac{(p+N)i}{ N}  } 4^{ \frac{r(p+N)j}{N}}
\Big(1+\frac{1}{\hat{t}} \Big)^{ \frac{p+N}{ N}  }  \hat{t}^{ \frac{p}{N}  }
 \frac{   M_{i+1}^{   r- \frac{p(N+2+\alpha_-)}{N} }    }{  k_0^{\frac{(p+N)\hat{\theta}}{N}}  }
\left(\dashint_{\Gamma_{2j-2}}\int_{ \Omega }w_{2j-2}^{   r  } \,dxd\tau\right)^{\frac{p+N}{ N}}\nonumber\\
& \leq\bar{C}\mathfrak{b}^{ \frac{p(p+N)}{ N(p-1)}  } 2^{ \frac{(p+N)i}{ N}  }  4^{ \frac{r(p+N)j}{N}}
 \Big(\hat{t}+\frac{1}{\hat{t}^{\frac{N}{p}}   } \Big)^{ \frac{p}{ N}  }
 \frac{  M_{i+1}^{   r- \frac{p(N+2+\alpha_-)}{N} }    }{  k_0^{\frac{(p+N)\hat{\theta}}{N}}  }
Y_{j-1}^{1+\frac{p}{ N}},
\end{align}
where the pure constant $\bar{C}>0$ depends only on \verb"data". This in conjunction with Lemma \ref{lem-iteration}
indicates that $Y_{j} \rightarrow 0$ as $j \rightarrow \infty$, provided that
$$
Y_0\le
\bar{ C}^{-\frac{N}{p}} 2^{-\frac{(p+N)i}{p}} 4^{-\frac{r(p+N)N}{p^2}}\Big(\hat{t}+\frac{1}{   \hat{t}^{\frac{N}{ p} }   } \Big)^{ -1 }
 \mathfrak{b}^{   -\frac{p+N}{p-1}   }
k_0^{\frac{\hat{\theta}(p+N)}{p}}M_{i+1}^\frac{p(N+2+\alpha_-)-Nr}{p},
$$
The above inequality can be achieved by taking
\begin{align*}
k_0\ge\max\bigg\{\bar{C}2^{\frac{i}{ \hat{\theta }}}
 4^{\frac{rN}{p \hat{\theta }}}
 \mathfrak{b}^{\frac{p}{\hat{\theta}(p-1) }  }
 \Big(\hat{t}+\frac{1}{  \hat{t}^{\frac{N}{p}}  } \Big)^{     \frac{p}{\hat{\theta}(p+N)}       }
\left(\dashint_{\hat{\Gamma}_i}\int_{ \Omega } n^{r} \,dxd\tau\right)^{   \frac{p}{\hat{\theta}(p+N)}         }
M_{i+1}^{\frac{  Nr- p(N+2+\alpha_-)  }{ \hat{\theta}(p+N)  }}\,   ,\,    K_f      \bigg\}.
\end{align*}
With this choice of $k_0$ and recalling the definition of $m$, we by $Y_j\rightarrow0$ have
\begin{align}\label{add4}
\underset{\Omega\times\hat{\Gamma}_{i}}{\rm ess\sup }~ n
\leq C2^{\frac{i}{ \hat{\theta }}}    \mathfrak{b}^{\frac{p}{\hat{\theta} (p-1) }  }
\Big(\hat{t}+\frac{1}{  \hat{t}^{\frac{N}{p}}  } \Big)^{     \frac{p}{\hat{\theta}(p+N)}       }
\left(\dashint_{\hat{\Gamma}_i}\int_{ \Omega } n^{r} \,dxd\tau\right)^{   \frac{p}{\hat{\theta}(p+N)}         }
M_{i+1}^{  m }
+K_f.
\end{align}
An application of Young's inequality to \eqref{add4} implies that for $i=0,1,2, \cdots$,
\begin{align*}
M_{i} \leq &\eta M_{i+1}
+C2^{   \frac{i}{ \hat{\theta }(1-m)}   }    \mathfrak{b}^{\frac{p}{   \hat{\theta}(p-1)(1-m)   }  }
\Big(\hat{t}+\frac{1}{  \hat{t}^{\frac{N}{p}}  } \Big)^{ \frac{p}{ \hat{\theta} (p+N)(1-m) }  }
\eta^{-\frac{m}{1-m}}\left(\dashint_{\hat{\Gamma}_i}\int_{ \Omega } n^{r} \,dxd\tau\right)^{   \frac{p}{\hat{\theta}(p+N)(1-m)}         }
+K_f.
\end{align*}
By using an iteration argument, we get
\begin{align}
M_{0}& \leq \eta^{i+1} M_{i+1}+
 C \mathfrak{b}^{\frac{p}{   \hat{\theta}(p-1)(1-m)   }  }
  \Big(\hat{t}+\frac{1}{  \hat{t}^{\frac{N}{p}}  } \Big)^{ \frac{p}{ \hat{\theta} (p+N)(1-m) }  }
\eta^{-\frac{m}{1-m}}\left(\dashint_{\hat{\Gamma}_0    }
\int_{ \Omega } n^{r} \,dxd\tau\right)^{   \frac{p}{\hat{\theta}(p+N)(1-m)}         }\nonumber\\
&\quad\times\sum_{k=0}^{i}(2^{   \frac{1}{ \hat{\theta }(1-m) }   }\eta)^{k} +K_f\sum_{k=0}^{i}\eta^{k}, \quad i=0,1,2, \ldots.
\end{align}
We choose $\eta=1/(2^{   \frac{1}{ \hat{\theta }(1-m)}+1})$ to deduce that the sum on the right-hand side
can be majored by a convergent series, and then take $i \rightarrow \infty$ to obtain
\begin{align}
\underset{\Omega\times(t_0-\hat{t}/2,t_0)}
{\rm ess~ sup} n
\leq C  \mathfrak{b}^{\frac{p}{   \hat{\theta}(p-1)(1-m)   }  }
\Big(\hat{t}+\frac{1}{  \hat{t}^{\frac{N}{p}}  } \Big)^{ \frac{p}{ \hat{\theta} (p+N)(1-m) }  }
\left(\dashint_{{t_0-\hat{t}} }^{ t_0 }
\int_{ \Omega } n^{r} \,dxd\tau\right)^{   \frac{p}{\hat{\theta}(p+N)(1-m)}         }+K_f.
\end{align}
This ends the proof.
\end{proof}

\medskip

\begin{proof}
[\bf Proof of Theorem \ref{th3}.]
Let us denote
$$\bar{\mathbf{a}}(\xi, s,x,t):=a(\xi, s,x,t)\xi-b(s,x,t)\nabla c-\tau n\mathbf{u},\qquad\xi\in\Bbb{R}^N,~ s \geq 0,~x\in\overline{\Omega},~t>0$$
and
$$\bar{f}(x,t):=f(n(x,t),x,t),\qquad x\in\Omega,~t>0.$$
Under this setting, \eqref{q1}$_1$ with the boundary condition \eqref{q1}$_3$ turns to be
 \begin{eqnarray*}
 \begin{cases}
    \begin{array}{llll}
       \displaystyle n_t=\nabla\cdot\bar{\mathbf{a}}(\nabla n, n,x,t)+\bar{f}(x,t), &&x\in\Omega,\,t>0,\\
        \displaystyle \bar{\mathbf{a}}(\nabla n, n,x,t) \cdot\nu=0, &&x\in\partial\Omega,\,t>0.
    \end{array}
 \end{cases}
\end{eqnarray*}
Based on the boundedness results established in Theorems \ref{th1}\&\ref{th2},
it is easy to verify that $\bar{\mathbf{a}}$ satisfies
$$\bar{\mathbf{a}}(\nabla n, n,x,t)\cdot \nabla n \ge a_1|\nabla n|^p-a_2,\qquad x\in\Omega,~t>0,$$
$$|\bar{\mathbf{a}}(\nabla n, n,x,t)| \le a_3|\nabla n|^{p-1}+a_4,\qquad x\in\Omega,~t>0$$
and
$$|\bar{f}(x,t)| \le a_5,\qquad x\in\Omega,~t>0$$
with $a_i>0$ ($i=1,2,3,4,5$).
According to \cite[Theorem 1.3, Chapter IV$\&$ Theorem 1.3, Chapter III]{D1993} and \cite[Lemma 2.3]{L1987},
we can directly claim that
$n$ is H\"{o}lder continuous in $\overline{\Omega} \times\big(t_0, \hat{T}\big]$ for any $t_0,\hat{T}\in (0,T)$
satisfying $t_0< \hat{T}$.
More specifically, for every pair of points $\left(x_1, t_1\right),\left(x_2, t_2\right) \in \overline{\Omega} \times\big[t_0+\varepsilon, \hat{T}\big]$,
\begin{align*}
| n\left(x_1, t_1\right) -n\left(x_2, t_2\right)|
 \leq C
\left(\left|x_1-x_2\right|+\left|t_1-t_2\right|^{  \frac{1}{p}  }\right)^\gamma.
\end{align*}
The constants $C>1$ and $\gamma\in(0,1)$ depend only upon $t_0,\|n\|_{L^\infty(\Omega \times(t_0, \hat{T}))}$,
$a_i$ ($1\le i\le5$) and the structure of $\partial \Omega$.

Furthermore,  we apply the regularity results established for degenerate parabolic equations, e.g., \cite[Theorem 1.1$'$, Chapter IX]{D1993}, \cite[Theorem 5.1, Chapter VIII]{D1993} and \cite[Proposition 4.1]{DZZ}
to obtain the H\"{o}lder continuity of $ \nabla n$ for the case of $p>2$.
When $p=2$, the global continuity of $\nabla n$ can be found in \cite[Theorem 1]{L1987}.
\end{proof}

\section{An upper bound independent of initial data}
\label{sec3}
This section is dedicated to building an upper bound of $n$ independent of initial data.
We let $\mathcal{L}_q$ denote the operator $-\Delta+1$ under homogenous Neumann boundary condition, that is, the sectorial operator is defined by
$$
\mathcal{L}_q \varphi:=-\Delta  \varphi+ \varphi
\quad \text { for }  \varphi \in D(\mathcal{L}_q):=\left\{\varphi\in W^{2,q}(\Omega): \partial_\nu  \varphi=0 \text { on } \partial \Omega\right\}.
$$
Then it is known from \cite{F1969,HW2005} that $\mathcal{L}$ possesses closed
fractional powers $\mathcal{L}_q^{\gamma}$ $(\gamma> 0)$, which
is a self-joint operator defined on the domain $D(\mathcal{L}_q^{\gamma})$.
\begin{lemma}\label{lem3.1}
{\rm (i)}~{\rm \cite[page 56]{HW2005}} Let $\gamma\in(0,1)$ and $1\le q<\ell\le\infty$.
The analytic semigroup $\left\{e^{-t\mathcal{L}}\right\}_{t \geq 0}$
{\rm (}which is independent of $q$ in the sense that $e^{-t \mathcal{L}_q} \varphi=e^{-t \mathcal{L}_{\ell}} \varphi$
whenever $\varphi \in L^{\ell}(\Omega) \cap L^q(\Omega)${\rm)} satisfies that
$$
\left\|\mathcal{L}^{\gamma} e^{-t\mathcal{L}} \varphi\right\|_{L^\ell(\Omega)}
\leq C t^{-\gamma-\frac{N}{2}( \frac{1}{q}-\frac{1}{\ell} )   } e^{-\lambda t}\|\varphi\|_{L^q(\Omega)},\qquad\forall~ t>0
$$
for all $\varphi \in L^q(\Omega)$, and with some $\lambda>0$.\\
{\rm (ii)}~{\rm \cite[Theorem II 14.1]{F1969}} For all $\alpha , \beta , \delta \in \mathbb{R}$ satisfying $\sigma < \gamma < \delta $, there is $C > 0$ such that
$$\| \mathcal{L}^{\gamma}\varphi \|_{L^{  2 }(\Omega)}
\leq C \| \mathcal{L}^{\delta}\varphi \|_{L^{  2 }(\Omega)}^{ \frac{\gamma-\sigma}{\delta -\sigma} }
\| \mathcal{L}^{\sigma}\varphi \|_{L^{  2 }(\Omega)}^{ \frac{\delta-\gamma}{\delta -\sigma} }\qquad { for~ all}~ \varphi \in D( \mathcal{L}^{\delta} ).$$
\end{lemma}

\begin{lemma}\label{lem3.2}
Let $(n,c)$ be a weak and global solution of \eqref{q1} and
$g(x,t):= g(n(x,t),x,t) {\rm ~~in~} \Omega\times(0,\infty). $
There are pure constants $C=C(\verb"data")>0$ and $\kappa=\kappa(\verb"data")>0$ with the property that
if there holds that
\begin{align}\label{ugK}
\|\mathbf{u}(\cdot,t)\|_{L^\infty(\Omega)}<K,\quad\|g(\cdot,t)\|_{L^{m}(\Omega)}\le K,\qquad \forall~t>t_1
\end{align}
with some $K>1$, $m>N$ and $t_1>0$, then we can find $t_2>t_1$ such that
$$\|\nabla c(\cdot, t)\|_{L^{\infty }(\Omega)}\le C K^{\kappa}+C,\qquad \forall~t>t_2.$$
\end{lemma}

\begin{proof}
[\bf Proof.]
{\bf Step 1.}
We arbitrarily select constants $\sigma\in[1/2,1)$, $q,\ell>1$ such that
\begin{align}\label{si}
q< \ell,\qquad\sigma+\frac{N}{2q}-\frac{N}{2\ell}<1,\qquad \sigma+\frac{N}{2m}-\frac{N}{2\ell}<1.
\end{align}
This step is used to show that if $M_q(\bar{s})=\|\mathcal{L}^{ \frac{1}{2}} c\|_{L^\infty((\bar{s},\infty);L^q(\Omega))}<\infty$ with some $\bar{s}>t_1$,
then there exists a time point $\hat{s}>\bar{s}$ and
a constant $C(\ell,q)>0$ depending only on $\ell$, $q$, $\sigma$ as well as $\Omega$
ensuring
\begin{align}\label{step1}
\|\mathcal{L}^{\sigma} c(\cdot, t)\|_{L^{ \ell }(\Om)}\le C(\ell,q)K+ C(\ell,q) K M_q(\bar{s}),
\qquad\forall~ t >\hat{s}.
\end{align}

Rewriting $c$ according to the variation-of-constants formula, it follows that $\mbox{for any } t >\bar{s}$,
\begin{align}\label{st11}
\|\mathcal{L}^{\sigma} c(\cdot, t)\|_{L^{ \ell }(\Omega)}
\leq &
\left\|\mathcal{L}^{\sigma} e^{-\mathcal{L}(t-\bar{s})} c(\cdot,\bar{s})\right\|_{L^{ \ell }(\Omega)}
+\int_{\bar{s}}^{t} \left\|\mathcal{L}^{\sigma} e^{-\mathcal{L}(t-s)} \mathbf{u}\cdot\nabla c(\cdot, s) \right\|_{L^{ \ell }(\Omega)}\,ds
\nonumber\\
&
+\int_{\bar{s} }^{t}\left\|\mathcal{L}^{\sigma} e^{-\mathcal{L}(t-s)} g(\cdot,s) \right\|_{L^{ \ell }(\Omega)}\,ds.
\end{align}
We employ Lemma \ref{lem3.1} (i) to get
\begin{align}\label{st12}
\|\mathcal{L}^{\sigma} e^{- \mathcal{L}(t-\bar{s})} c(\cdot,\bar{s})\|_{L^{ \ell }(\Omega)}
 \le  C (t- \bar{s})^{-\sigma-\frac{N}{2}\big(\frac{2}{\ell+1}-\frac{1}{\ell}\big)} e^{-\lambda (t-\bar{s})}
\left\|c(\cdot,\bar{s}) \right\|_{L^{\frac{\ell+1}{2}}(\Omega)},
\qquad\forall~ t>\bar{s}.
\end{align}
This implies the existence of $\hat{s}>\bar{s}$ such that
\begin{align}\label{st13}
\|\mathcal{L}^{\sigma} e^{- \mathcal{L}(t-\bar{s})} c(\cdot,\bar{s})\|_{L^{ \ell }(\Omega)}
 \le  K,
\qquad\forall~ t>\hat{s}.
\end{align}
Utilizing Lemma \ref{lem3.1} (i) again and applying H\"{o}lder's inequality, we obtain by \eqref{ugK} that
\begin{align}
&\quad\int_{\bar{s}}^{t}\left\|\mathcal{L}^{\sigma} e^{-\mathcal{L}(t-s)}\mathbf{u}(\cdot, s)
\cdot \nabla c(\cdot, s)\right\|_{L^{ \ell }(\Omega)} \,ds \nonumber \\
&\le
C\int_{\bar{s}}^{t}
(t-s)^{-\sigma-\frac{N}{2}(\frac{1}{q} - \frac{1}{\ell}   ) } e^{-\lambda(t-s)}
\| \mathbf{u}(\cdot, s) \cdot \nabla c(\cdot, s) \|_{L^{ q  }(\Omega)} \,ds \nonumber \\
&\le
C\int_{\bar{s}}^{t}  (t-s)^{-\sigma-\frac{N}{2q}+\frac{N}{2\ell}} e^{-\lambda(t-s)}
\|\mathbf{u}(\cdot, s)\|_{L^{\infty}(\Omega)} \|\mathcal{L}^{\frac{1}{2}} c(\cdot, s) \|_{L^{ q }(\Omega)} \,ds
\nonumber \\
&\le C K M_q(\bar{s})
\qquad\mbox{for any } t >\bar{s}
\end{align}
because of \eqref{ugK} and \eqref{si}. Similarly, as above, there holds that
\begin{align}\label{st14}
\int_{ \bar{s} }^{t}\left\|\mathcal{L}^{\sigma} e^{-\mathcal{L}(t-s)}g(\cdot, s)\right\|_{L^{ \ell }(\Omega)} \,ds
&\le
C\int_{\bar{s}}^{t}
(t-s)^{-\sigma-\frac{N}{2}(  \frac{1}{m}- \frac{1}{\ell}  ) } e^{-\lambda (t-s)}
\| g(\cdot, s) \|_{L^{ m  }(\Omega)}  \,ds \nonumber \\
&\le C K
\qquad\mbox{for any } t >\bar{s}.
\end{align}
A combination of \eqref{st11}--\eqref{st14} directly leads to our desired estimate \eqref{step1}.

{\bf Step 2}.
By testing \eqref{q1}$_2$ with $ c$ and integrating by parts, we have
\begin{equation}\label{st21}
\begin{split}
\frac{1}{2}\frac{\mathrm{d}}{\mathrm{d}t} \int_{\Omega}c^{2}\,dx
 +\int_{\Omega}|\nabla c  |^{2}\,d x
 +\int_{\Omega}c^{2}\,d x
\leq \int_{\Omega} g c \,d x
\leq \frac{1}{2}\int_{\Omega}c^{2}\,dx+\frac{1}{2}\int_{\Omega}g^2\,dx,\qquad\forall~ t>0,
\end{split}
\end{equation}
where we used $\nabla\cdot \mathbf{u}=0$ in $\Omega\times(0,T)$
and $c_{\nu}=\mathbf{u}=0$ on $\partial\Omega\times(0,T)$.
In view of \eqref{ugK}, we deduce from H\"{o}lder's inequality that
\begin{equation}\label{st22}
\begin{split}
&\quad\frac{\mathrm{d}}{\mathrm{d} t} \int_{\Omega}c^{2} \,dx+\int_{\Omega}c^{2}\,dx
\leq K^2,\qquad\forall~ t>t_1.
\end{split}
\end{equation}
An argument of ODI (ordinary differential inequalities) applied to the above display ensures the existence $s_0>t_1$ such that
\begin{align}\label{st22'}
\|c(\cdot,t)\|^2_{L^2(\Omega)}\leq K^2+1,\qquad\forall~ t>s_0.
\end{align}
By testing \eqref{q1}$_2$ with $\mathcal{L} c$,
 we derive from Young's inequality that
\begin{align}\label{st23}
\frac{\mathrm{d}}{\mathrm{d}t} \int_{\Omega}|\mathcal{L}^{\frac{1}{2}} c|^2\,dx
 +\int_{\Omega}|\mathcal{L} c|^2\,dx
&\leq \int_{\Omega}\mathbf{u} \cdot \nabla c\mathcal{L}c\,dx
+\int_{\Omega} g\mathcal{L} c\,dx \nonumber\\
&\leq   \frac{1}{2}\|\mathcal{L}c\|^2_{L^{ 2  }(\Omega)}
+\|\mathbf{u}\|^2_{L^\infty(\Omega)} \| \nabla c\|_{L^{ 2  }(\Omega)}^2
+\|g\|^2_{L^{ 2 }(\Omega)},\qquad\forall~ t>s_0,
\end{align}
where Lemma \ref{lem3.1} (ii) along with \eqref{st22'} tells that
\begin{align*}
\|\nabla c \|^{ 2 }_{L^{ 2 }(\Omega)}
&\le C \|\mathcal{L}^{\frac{1}{2}} c  \|^{ 2 }_{L^{ 2 }(\Omega)}
\le
C \|\mathcal{L} c \|_{L^{  2 }(\Omega)}
\|c \|_{L^{  2 }(\Omega)}  \\
&\le
\frac{1}{4K^2}\|\mathcal{L} c \|^{ 2  }_{L^{  2 }(\Omega)}
+CK^2\|c \|^{ 2 }_{L^{ 2 }(\Omega)}
\le\frac{1}{4K^2}\|\mathcal{L} c \|^{ 2  }_{L^{  2 }(\Omega)}
+C K^4+CK^2.
\end{align*}
Summarizing the two above inequalities and referring back to \eqref{ugK} result in the following estimate,
\begin{equation}\label{u1mm}
\begin{split}
&\quad\frac{\mathrm{d}}{\mathrm{d} t} \int_{\Omega}|\mathcal{L}^{\frac{1}{2}} c|^2 \,dx
 +\frac{1}{4}\int_{\Omega}|\mathcal{L} c|^2\,dx
\leq CK^6+CK^4,\qquad \forall~t>s_0.
\end{split}
\end{equation}
By using arguments of ODI again, we can find $s_1>s_0$ fulfilling
$$\|\nabla c(\cdot, t)\|_{L^{ 2 }(\Omega)}\le C K^3+C K^2,\qquad \forall~ t>s_1. $$

{\bf Step 3}. In this step, we employ results obtained in the above two steps to achieve our final goal via 
an iterative process.
Let us take $\sigma=\frac{1}{2}$ in \eqref{si}, $q_i=2(\frac{N}{N-1})^{i-1}$ and $\ell_i=\frac{Nq_i}{N-1}$ with $i\in\mathbb{N}^+$.
By verifying
$$
\frac{1}{2}+\frac{N}{2q_i}-\frac{N}{2\ell_i}<1, \quad \frac{1}{2}+\frac{N}{2m}-\frac{N}{2\ell_i}<1
$$
thanks to the assumption $m>N$,
we can derive from \eqref{step1} that if $M_{q_i}(s_{i})<\infty$, then there exists $s_{i+1}>s_i$ such that
\begin{align}\label{529a1}
M_{q_{i+1}}(s_{i+1})=M_{\ell_i}(s_{i+1})=\|\mathcal{L}^{ \frac{1}{2} } c\|_{L^\infty((s_{i+1},\infty);L^{\ell_i}(\Omega))}<\infty. 
\end{align}
The bound for $\| \mathcal{L}^{\frac{1}{2}} c\|_{L^2(\Omega)}$ presented in {Step 2} serves as the starting point for our intended iteration. This combined with \eqref{529a1} guarantees the existence of increasing sequences $\{s_{i}\}_{i\in\mathbb{N}^+ }$, $\{C_{i}\}_{i\in\mathbb{N}^+ }$ and $\{\kappa_{i}\}_{i\in\mathbb{N}^+ }$ ensuring
\begin{align}\label{st31}
M_{q_i}(s_{i})=\|\mathcal{L}^{ \frac{1}{2} } c\|_{L^\infty((s_{i},\infty);L^{q_i}(\Omega))}
\leq C_{i}K^{\kappa_i}+C_{i}.
\end{align}
With arbitrarily fixing $\hat{q}=2(\frac{N}{N-1})^{i_0-1}>N$,
we further select $\hat{\ell}>\max\{\hat{q},m\}$ and take $\hat{\sigma}\in(0,1)$ satisfying 
$$\frac{N}{2\hat{\ell}} + \frac{1}{2}<\hat{\sigma}< \frac{N}{2\hat{\ell}}+1- \frac{N}{2\min\{\hat{q},m\}}.$$
These choices ensure that 
$$\hat{\sigma}+\frac{N}{2\hat{q}}-\frac{N}{2\hat{\ell}}<1,\quad \hat{\sigma}+\frac{N}{2m}-\frac{N}{2\hat{\ell}}<1,
\quad (2\hat{\sigma}-1)\hat{\ell}>N. $$
Since $M_{ \hat{q} }(s_{i_0})=\|\mathcal{L}^{ \frac{1}{2} } c\|_{L^\infty((s_{i_0},\infty);L^{\hat{q}}(\Omega))}
\leq  C_{i_0}K^{\kappa_{i_0}}+C_{i_0}$
by virtue of \eqref{st31},
we can apply \eqref{step1} again to conclude that
\begin{align}
\|\mathcal{L}^{\sigma} c(\cdot, t)\|_{L^{\hat{\ell}}(\Omega)}\le CK+ C K M_{ \hat{q} }(s_{i_0}),\qquad \forall ~t>t_2
\end{align}
with some $t_2>s_{i_0}$.
This in conjunction with the embedding $D(\mathcal{L}_{\hat{\ell}}^{\sigma})\hookrightarrow W^{1,\infty}(\Omega)$ clearly implies the desired estimate.
\end{proof}

\medskip

\begin{proof}
[\bf Proof of Theorem \ref{th4}.]
Based on Lemma \ref{lem3.2}, we can directly arrive at the desired estimates by
\eqref{th11} and \eqref{th21}.
\end{proof}

\section{Applications}

\label{sec4}

\subsection{Asymptotic stability in chemotaxis models}

{\bf Example (A)} Our first example elucidates the application of Theorem \ref{th4} in investigating
the large-time behaviours for a chemotaxis model involving nonlinear signal production, given as follows:
 \begin{eqnarray}\label{example1}
 \begin{cases}
    \begin{array}{llll}
       \displaystyle n_t=\nabla\cdot\big(\nabla n-n\nabla c\big), &&x\in\Omega,\,t>0,\\
        \displaystyle  c_{t}=\Delta c-c +n^{\sigma}, &&x\in\Omega,\,t>0,\\
        \displaystyle (n,c)|_{t=0}=(n_0,c_0),&&x\in\Omega
    \end{array}
 \end{cases}
\end{eqnarray}
under the homogenous Neumann boundary condition for both components.
The initial data $\left(n_0, c_0\right)$ fulfills
\begin{align}\label{1initial}
\left\{\begin{array}{l}
n_0 \in C^\omega(\overline{\Omega})~(0<\omega<1), ~n_0 \geq 0, ~n_0 \not \equiv 0 \quad \text {on } \overline{\Omega}, \\
c_0 \in W^{1, \infty}(\Omega), ~c_0 \geq 0  \quad \text {on } \overline{\Omega} .
\end{array}\right.
\end{align}

As a prerequisite for applying Theorem \ref{th4},
we need to prove that in the model \eqref{example1},
the $L^{q}$--norm of $g(x,t)= n^{\sigma}(x,t)$ becomes independent of the initial data after some time.
The proof of this part will be placed in the Appendix, and in the next proposition. We will directly utilize the result established in Lemma \ref{A1} to illustrate how we use Theorem \ref{th4} to achieve our final
goal of determining asymptotic stability.

\begin{proposition}
Let $\Omega\subseteq \Bbb{R}^N$ $(N\ge2)$ be a bounded domain with smooth boundary and $0<\sigma<2/N$.
Then there is $M_*>0$ only depending on $\Omega$, $N$ and $\sigma$ such that for any initial data $(n_0,c_0)$
satisfying \eqref{1initial} and $\dashint_{\Omega}n_0\,dx= M<M_*$, \eqref{example1} admits a classical and globally bounded solution $(n,c)$, which converges to $(M,M^{\sigma})$ in $L^\infty(\Omega)$ exponentially.
\end{proposition}

\begin{proof}
[\bf Proof]
We begin the proof by assuming $M_*\le1$ and observing
\begin{align}\label{529a3}
\dashint_{\Omega}n(\cdot, t)\,dx\equiv \dashint_{\Omega}n_0 \,dx=M<M_*,\qquad \forall~t>0.
\end{align}
Invoking Lemma \ref{A1} with an arbitrarily selected number $q>\frac{N}{2}$,  one can find $t_1>0$ such that
\begin{align}\label{N}
\|n(\cdot,t)\|_{L^q(\Omega)}\leq C_1,\qquad \forall~t>t_1
\end{align}
with $C_1=C_1(\sigma,N,\Omega)>0$.
Lemma \ref{lem3.1} provides a time point $t_2>t_1$ ensuring
\begin{align}\label{ex0}
\|\nabla c(\cdot, t)\|_{L^{\infty }(\Omega)}\le C_2,\qquad \forall~t>t_2
\end{align}
with $C_2=C_2(\sigma,N,\Omega)>0$.
An application of Theorem \ref{th4} immediately gives an upper bound of $n$ only dependent on $\sigma$, $N$ and $\Omega$.
Having this upper bound at hand, we further exploit Theorem \ref{th3} to find a pure constant $C_3=C_3(\sigma, N,\Omega)>0$ fulfilling
\begin{align}\label{ex0'}
\|n(\cdot, t)\|_{W^{1,\infty}(\Omega)}\le C_3,\qquad \forall~t>t_2.
\end{align}

With \eqref{ex0} and \eqref{ex0'} explicitly exhibiting desired estimates,
 we proceed in the remaining portion to demonstrate how a smallness condition and
the mass conservation can ensure the asymptotic stability of the model.
Testing \eqref{example1}$_1$ by $n^{\hat{q}-1}$ with fixed $\hat{q}>\frac{N(1-\sigma)}{\sigma}$ and integrating by parts show that
\begin{align} \label{ex1}
\frac{1}{\hat{q}}\frac{\mathrm{d}}{\mathrm{d}t}\int_{\Omega}n^{\hat{q}}\,dx
+\frac{ 4(\hat{q}-1) }{ \hat{q}^2 }\int_{\Omega}|\nabla n^{\frac{\hat{q}}{2}} |^2\,dx  \notag
=& (\hat{q}-1)
\int_{\Omega}n^{\hat{q}-1}\nabla n\cdot \nabla c \,dx \notag \\
\leq& \frac{\hat{q}-1}{2}\int_{\Omega}n^{\hat{q}-2}|\nabla n|^2  \,dx
+
\frac{\hat{q}-1}{2}\int_{\Omega}n^{\hat{q}}|\nabla c|^2 \,dx,\qquad \forall~t>t_2.
\end{align}
The last integral can be estimated by \eqref{ex0} as follows,
\begin{align}
\int_{\Omega}n^{\hat{q}}|\nabla c|^2\,dx\le C^2_2\int_{\Omega}n^{\hat{q}}\,dx
= C^2_2\| n^{\frac{\hat{q}}{2}} \|^2_{L^2(\Omega)}.
\end{align}
Utilizing Young's inequality as well as the Gagliardo-Nirenberg inequality to govern the above display, we thereby obtain
\begin{align}\label{529a2}
 (C^2_2+1)\| n^{\frac{\hat{q}}{2}} \|^2_{L^2(\Omega)}
&\le C_3\Big(\|\nabla n^{\frac{\hat{q}}{2}} \|^{\theta}_{L^2(\Omega)}\| n^{\frac{\hat{q}}{2}}\|^{1-\theta}_{L^{\frac{2}{\hat{q}}}(\Omega)}
+\| n^{\frac{\hat{q}}{2}}\|_{L^{\frac{2}{\hat{q}}}(\Omega)}\Big)^2\nonumber\\
&\le \frac{2(\hat{q}-1) }{ \hat{q}^2   } \int_{\Omega}|\nabla n^{\frac{\hat{q}}{2}} |^2\,dx
+2\big(\hat{q}^{\frac{2}{\theta}}C_3\big)^{  \frac{1}{1-\theta} } (\hat{q}-1)^{-\frac{1}{\theta(1-\theta)}}
\| n^{\frac{\hat{q}}{2}}\|^2_{L^{\frac{2}{\hat{q}}}(\Omega)}
+2C_3\| n^{\frac{\hat{q}}{2}}\|^2_{L^{\frac{2}{\hat{q}}}(\Omega)}
\end{align}
with $C_3=C_3(\hat{q},N,\Omega)>0$ and
$$
\theta:=\frac{\frac{\hat{q}}{2}-\frac{1}{2}}{\frac{\hat{q}}{2}+\frac{1}{N}-\frac{1}{2}}.
$$
Thus, by taking $C_4=2\big(\hat{q}^{\frac{2}{\theta}}C_3\big)^{  \frac{1}{1-\theta} } (\hat{q}-1)^{-\frac{1}{\theta(1-\theta)}}+2C_3$, we have from \eqref{529a3} and \eqref{ex1}--\eqref{529a2} that
\begin{align*}
\frac{1}{\hat{q}}\frac{\mathrm{d}}{\mathrm{d} t}\int_{\Omega}n^{ \hat{q} } \,dx
+\int_{\Omega}n^{ \hat{q} }\,dx  \le C_4 M^{ \hat{q} }, \qquad\forall~ t>t_2.
\end{align*}
Applying a result of ODI, we can find $t_3>t_2$ such that
\begin{align*}
\int_{\Omega}n^{ \hat{q} } \,dx  \le 2\hat{q}C_4 M^{ \hat{q} }, \qquad\forall~ t>t_3.
\end{align*}
By invoking the Gagliardo-Nirenberg inequality again and using \eqref{ex0'} as well as \eqref{529a3}, it follows that
\begin{align}\label{ex2}
\|n(\cdot,t)\|_{L^\infty(\Omega)}\le C_5\|n(\cdot,t)\|_{W^{1,\infty}} ^{ \frac{N}{N+\hat{q}}  }
\|n(\cdot,t)\|_{L^{ \hat{q}  }(\Omega)} ^{\frac{\hat{q}}{N+\hat{q}}}
 +C_5\|n(\cdot,t)\|_{L^{  \hat{q} }(\Omega)}\le C_6(M^{\frac{\hat{q}}{N+\hat{q}}}+M),\qquad\forall~ t>t_3
\end{align}
with $C_5=C_5(\hat{q},N,\Omega)>0$, $C_6=C_6(\hat{q},N,\sigma,\Omega)>0$.
Multiplying \eqref{example1}$_1$ with $n-M$ and integrating by parts, we get
\begin{align}\label{ex3}
\frac{1}{2}\frac{\mathrm{d}}{\mathrm{d} t}\int_{\Omega}(n-M)^2\,dx+\int_{\Omega}|\nabla n|^2\,dx  \notag
=& -\int_{\Omega}n\nabla n\cdot \nabla c\,dx \notag \\
\leq& \frac{1}{2}\int_{\Omega}|\nabla n|^2\,dx +
 \frac{1}{2}\int_{\Omega}n^{2}|\nabla c|^2\,dx, \qquad\forall~ t>t_3,
\end{align}
  where the Poincar\'{e} inequality indicates
  $$\int_{\Omega}|\nabla n|^2 \,dx\ge C_p\int_{\Omega}|n-M|^2 \,dx.$$
To estimate the last factor in (\ref{ex3}) we majorise the integrand by \eqref{ex2},
  \begin{align*}
\int_{\Omega} n^2|\nabla c|^2\,dx
\le \|n(\cdot,t)\|_{L^\infty(\Omega)}^2\int_{\Omega}|\nabla c|^2\,dx
\le 4C_6^2 \big(M^{\frac{2\hat{q}}{\hat{q}+N} }+M^2\big)\int_{\Omega}|\nabla c|^2\,dx,\qquad\forall~ t>t_3.
  \end{align*}
We combine the above three inequalities and obtain
     \begin{align}\label{ex3'}
   \frac{\mathrm{d}}{\mathrm{d} t}\int_{\Omega} (n-M)^2\,dx+C_p\int_{\Omega}(n-M)^2\,dx
   \leq 4C_6^2 \big(M^{\frac{2\hat{q}}{\hat{q}+N} }+M^2\big)\int_{\Omega} |\nabla c|^2\,dx ,  \qquad\forall~ t>t_3.
  \end{align}
Testing \eqref{example1}$_2$ with $c-M^{\sigma}$, integrating by parts and utilizing Young's inequality, we have
\begin{align}\label{z91}
\frac{\mathrm{d}}{\mathrm{d} t}
\int_\Omega (c-M^{\sigma})^2\,dx
+2\int_\Omega|\nabla c|^2\,dx+\int_\Omega(c-M^{\sigma})^2\,dx
&\le\int_\Omega (n^\sigma-M^{\sigma})^2 \,dx, \qquad\forall~ t>t_3.
\end{align}
For the point $(\bar{x}, \bar{t}) \in \Omega \times(0, \infty)$ with $n(\bar{x}, \bar{t}) \leq \frac{M}{2}$,
it can be verified that
$$
\begin{aligned}
\left|n^\sigma-M^\sigma\right|& \leq\left|n-M\right|^{\sigma}=\left|n-M\right|^{\sigma-1}\left|n-M\right|\leq\left(\frac{M}{2}\right)^{\sigma-1}\left|n-M\right| .
\end{aligned}
$$
For the point $(\bar{x}, \bar{t}) \in \Omega \times(0, \infty)$ satisfying $n(\bar{x}, \bar{t})>\frac{M}{2}$, it follows by the mean value theorem that
$$
\begin{aligned}
\left|n^\sigma-M^\sigma\right|  =\left|\tilde{h}\left(n\right)-\tilde{h}\left(M\right)\right|\leq \tilde{h}^{\prime}\left(n-\chi n+\chi M\right)\left|n-M\right|
\end{aligned}
$$
for some $\theta \in(0,1)$ with $\tilde{h}(s):=s^{\sigma}$ on $\left(\frac{M}{4}, \infty\right)$. Notice the monotonic decreasing property of $\tilde{h}^{\prime}(s)=\sigma s^{\sigma-1}$ and the fact $n-\theta n+\theta M>\frac{M}{2}$ ensured by $n(\bar{x}, \bar{t})>\frac{M}{2}$.
We have
$$
\left|n^\sigma-M^\sigma\right|\leq\sigma \Big(\frac{M}{2}\Big)^{\sigma-1}\left|n-M\right|.
$$
Thus, by \eqref{ex3}, there holds that
\begin{align}\label{ex4}
\frac{\mathrm{d}}{\mathrm{d} t}\int_\Omega (c-M^{\sigma})^2\,dx
+2\int_\Omega|\nabla c|^2\mathrm{d}x+\int_\Omega(c-M^{\sigma})^2\,dx
\le\left(\frac{M}{2}\right)^{2\sigma-2}\int_\Omega (n-M)^{2} \,dx, \quad\forall~ t>t_3.
\end{align}
We derive from \eqref{ex3'} and \eqref{ex4} that
\begin{align*}
&\frac{\mathrm{d}}{\mathrm{d} t}\left(\int_{\Omega}(n-M)^2\,dx
+2C_6^2 \big(M^{\frac{2\hat{q}}{\hat{q}+N} }+M^{2} \big) \int_\Omega (c-M^{\sigma})^2\,dx\right)\\
&\quad+C_p\int_{\Omega}(n-M)^2\,dx
+2C_6^2 \big(M^{\frac{2\hat{q}}{\hat{q}+N} }+M^{2} \big) \int_\Omega(c-M^{\sigma})^2\,dx \\
&\le
2C_6^2\big(M^{ \frac{2\hat{q}}{\hat{q}+N} } +M^{2} \big)\left(\frac{M}{2}\right)^{2\sigma-2}\int_\Omega (n-M)^{2} \,dx, \qquad\forall~ t>t_3.
\end{align*}
Since the choice $\hat{q}>\frac{N(1-\sigma)}{\sigma}$ ensures $ \frac{2\hat{q}}{\hat{q}+N} >2-2\sigma$, we can find a positive number $M_*\leq 1$ such that
 $$ 4C_6^2 \big(M_{*}^{\frac{2\hat{q}}{\hat{q}+N} }+M_{*}^2\big) \left(\frac{M_{*}}{2}\right)^{2\sigma-2}\le \frac{C_p}{2}.$$
 Then for any $t>t_3$,
\begin{align*}
&\frac{\mathrm{d}}{\mathrm{d} t}\left(\int_{\Omega}(n-M)^2\,dx
+2C_6^2 \big(M^{\frac{2\hat{q}}{\hat{q}+N} }+M\big) \int_\Omega (c-M^{\sigma})^2\,dx\right)\\
&\quad+\frac{C_p}{4}\int_{\Omega}(n-M)^2\,dx
+2C_6^2 \big(M^{\frac{2\hat{q}}{\hat{q}+N} }+M\big) \int_\Omega(c-M^{\sigma})^2\,dx \leq 0,
\end{align*}
provided $M\le M_{*}$. An application of the Gr\"{o}nwall inequality infers
 $$
\left\|n(\cdot, t)-M\right\|_{L^2(\Omega)}^2+\left\|c(\cdot, t)-M^{\sigma}\right\|_{L^2(\Omega)}^2
\leq C_6 e^{-C_7\left(t-t_3\right)}, \qquad t>t_3
$$
with $C_6,C_7>0$ depending on $M,\sigma,\Omega,N$. Combine this with the Gagliardo-Nirenberg inequality to get
\begin{align}\label{ex6}
\left\|n(\cdot, t)-M\right\|_{L^{\infty}(\Omega)} &
 \leq C_8\|\nabla n(\cdot, t)\|_{L^{\infty}(\Omega)}^{\frac{N}{N+2}}
 \left\|n(\cdot, t)-M\right\|_{L^2(\Omega)}^{\frac{2}{N+2}}
 +C_8\left\|n(\cdot, t)-M\right\|_{L^2(\Omega)} \nonumber\\
& \leq C_9
\left(e^{-\frac{C_7\left(t-t_3\right)}{N+2}}+e^{-\frac{C_7\left(t-t_3\right)}{2}}\right), \qquad t>t_3+2
\end{align}
with $C_8=C_8(\Omega,N)$, $C_9=C_9(M,\sigma,\Omega,N)>0$.
Carrying out the same arguments as resulting in \eqref{ex6}, we also can derive
$$
\left\|c(\cdot, t)-M^{\sigma}\right\|_{L^{\infty}(\Omega)} \leq C_{10} e^{-\frac{C_{7}\left(t-t_3\right)}{N+2}}, \qquad t>t_3+2
$$
for some $C_{10}=C_{10}(M,\sigma,\Omega,N)>0$.
This along with \eqref{ex6} completes the proof.
\end{proof}

\noindent{\bf Example (B)}
The second example is concerned with the constant equilibrium for a chemotaxis-Navier-Stokes model under the two-dimensional setting:
 \begin{eqnarray}\label{example2}
 \begin{cases}
    \begin{array}{llll}
       \displaystyle n_t+\mathbf{u}\cdot\nabla n=\nabla\cdot\big(D(n)\nabla n-S(n)\nabla c\big)+rn-\mu n^{1+\gamma}, &&x\in\Omega,\,t>0,\\
        \displaystyle  c_{t}+\mathbf{u}\cdot\nabla c=\Delta c-c +n, &&x\in\Omega,\,t>0,\\
           \displaystyle  \mathbf{u}_{t}+\mathbf{u}\cdot\nabla \mathbf{u}=\Delta \mathbf{u}+\nabla P +n\nabla\Phi, &&x\in\Omega,\,t>0,\\
        \displaystyle \nabla \mathbf{u}=0, &&x\in\partial\Omega,\,t>0,\\
        \displaystyle (n,c,\mathbf{u})|_{t=0}=(n_0,c_0,\mathbf{u}_0),&&x\in\Omega,
    \end{array}
 \end{cases}
\end{eqnarray}
where species density $n$ as well as nutrient concentration $c$ has no-flux boundary conditions on $\partial\Omega$, and the fluid $\mathbf{u}$ fulfills the homogenous Dirichlet condition. The initial data $\left(n_0, c_0, \mathbf{u}_0\right)$ satisfies
\begin{align}\label{2initial}
\left\{\begin{array}{l}
n_0 \in C^\omega(\overline{\Omega})~(0<\omega<1), ~n_0 \geq 0, ~n_0 \not \equiv 0  \quad \text {on } \overline{\Omega}, \\
c_0 \in W^{1, \infty}(\Omega), ~c_0 \geq 0 \quad \text {on } \overline{\Omega},\\
\mathbf{u}_0 \in L^\infty\left(\Omega ; \mathbb{R}^2\right) \text { and } \nabla \cdot \mathbf{u}_0=0 \quad \text {in } \mathcal{D}^{\prime}(\Omega).
\end{array}\right.
\end{align}
The nonlinearities $D,S\in C^2([0,\infty))$ satisfy
\begin{align}\label{cDS}
D(n)\ge a_0 (n+1)^{\alpha},\qquad
0\le S(n)\le b_0 n(n+1)^{\beta-1}
\end{align}
with $\alpha,\gamma\ge0$, $\beta\in\Bbb{R}$ and $a_0,b_0>0$.
On basis of the energy development of $(n,c,\mathbf{u})$ as exhibited in Lemma \ref{A2},
we use the following proposition to establish the large-time behaviours of solutions.
\begin{proposition}
Let $\Omega\subseteq \Bbb{R}^2$ be a bounded domain with smooth boundary and $\gamma>1$. Assume that $D$ and $S$ satisfy \eqref{cDS} with $\alpha\ge 0$ and $\beta\leq\gamma$. Then there is $\mu^*=\mu^*(a_0,b_0,\alpha,\beta,r,N,\Omega)>0$ such that whenever $\mu>\mu^*$, for any initial data $(n_0,c_0,\mathbf{u}_0)$ with \eqref{2initial}  the system \eqref{example2} possesses a globally bounded and classical solution $(n,c,\mathbf{u})$ converging to $\left((\frac{r}{\mu})^{\frac{1}{\gamma}},(\frac{r}{\mu})^{\frac{1}{\gamma}},0\right)$
in $L^\infty(\Omega)$ exponentially.
\end{proposition}
\begin{proof}
[\bf Proof]
According to Lemma \ref{A2}, we identify a number $\mu_*>0$ such that the condition $\mu>\mu_*$ guarantees the global
classical solvability to \eqref{2initial} and ensures the existence of a time point $t_1>0$ fulfilling
 \begin{align}\label{N'}
\|n(\cdot,t)\|_{L^2(\Omega)}\le C_1,\qquad\forall~ t>t_1
\end{align}
with $C_1>0$ depending only on $a_0,b_0,\alpha,\beta,\gamma,r,N$ and $\Omega$.
Furthermore, Theorem \ref{th4} provides such $C_2>0$ and $t_2>t_1$ that
 \begin{align}\label{N'}
\|n(\cdot,t)\|_{L^\infty(\Omega)}\le C_2+C_2\left(\frac{r}{\mu}\right)^{\frac{1}{\gamma}},\qquad \forall~t>t_2,
\end{align}
as long as $\mu>\mu_*$.
  Let us define $\chi:=\big(\frac{r}{\mu}\big)^{\frac{1}{\gamma}}$ and
$$
H(s):=s-\chi-\chi \ln \left(\chi^{-1} s\right)~ \text { for } s>0 .
$$
By the Taylor expansion, there exists $\theta=\theta(s,\chi) \in(0,1)$ such that
$$
\begin{aligned}
H(s) =H(\chi)+H^{\prime}(\chi)(s-\chi)+\frac{H^{\prime \prime}\big(\theta s+(1-\theta) \chi\big)}{2}(s-\chi)^2=\frac{\chi}{2(\theta s+(1-\theta) \chi)^2}(s-\chi)^2
\end{aligned}
$$
for all $s>0$. Thus, it clearly holds that $H(s) \geq 0$ and
\begin{align}\label{ex21}
\lim _{s \rightarrow \chi} \frac{H(s)}{(s-\chi)^2}
=\lim _{s \rightarrow \chi} \frac{\chi}{2(\theta s+(1-\theta) \chi)^2}=\frac{1}{2 \chi} .
\end{align}

We have by \eqref{example2}$_1$ and \eqref{cDS} that
\begin{align}\label{ex22}
\frac{\mathrm{d}}{\mathrm{d} t}\int_{\Omega}H\big(n(x,t)\big)\,dx
=&-\chi \int_\Omega
\frac{D(n)|\nabla n|^2}{n^2}\,dx
+\chi\int_\Omega
\frac{S(n)\nabla c \cdot\nabla n}{n^2}\,dx-\mu\int_\Omega
(n-\chi)\Big(n^{\gamma}-\frac{r}{\mu}\Big)\,dx\nonumber\\
\le &-\frac{\chi a_0}{2}\int_\Omega
\frac{|\nabla n|^2}{n^2}\,dx
+\frac{\chi b^2_0}{2a_0}\int_\Omega
(n+1)^{2\beta-\alpha-2 }|\nabla c|^2\,dx \nonumber\\
&-\mu\int_\Omega
(n-\chi)\Big(n^{\gamma}-\frac{r}{\mu}\Big)\,dx,\qquad \forall~t>t_2.
\end{align}
It can be checked by simple calculations that
\begin{align*}
(n-\chi)
\Big(n^{\gamma}-\frac{r}{\mu}\Big)=(n-\chi)(n^{\gamma}-\chi^{\gamma})
\ge\chi^{\gamma-1}(n-\chi)^2,
\end{align*}
and hence
\begin{align*}
-\mu\int_\Omega (n-\chi)\Big(n^{\gamma}-\frac{r}{\mu}\Big)\,dx
\le-\mu\chi^{\gamma-1}\int_\Omega(n-\chi)^2\,dx=-\frac{r}{\chi}\int_\Omega(n-\chi)^2\,dx.
\end{align*}
Substituting this to \eqref{ex22} yields
\begin{align}\label{ex23}
\frac{\mathrm{d}}{\mathrm{d} t}\int_{\Omega}H\big(n(x,t)\big)\,dx
\le \chi C_3\int_\Omega|\nabla c|^2\,dx
&-\frac{r}{\chi}\int_\Omega(n-\chi)^2\,dx,\qquad \forall~t>t_2
\end{align}
with $C_3:=\frac{ b^2_0}{2a_0}\big((C_2+C_2\chi+1)^{2\beta-\alpha-2 }+1\big)$. By \eqref{example2}$_2$,
\begin{align}\label{ex24}
\frac{\mathrm{d}}{\mathrm{d} t}\int_\Omega (c-\chi)^2\,dx
\le-2\int_\Omega|\nabla c|^2\,dx-\int_\Omega(c-\chi)^2\,dx+\int_\Omega(n-\chi)^2\,dx,
\qquad \forall~t>t_2.
\end{align}
By virtue of \eqref{ex23} and \eqref{ex24}, one sees that
\begin{align}\label{ex25}
&\quad\frac{\mathrm{d}}{\mathrm{d} t}\left(\int_{\Omega}H\big(n(x,t)\big)\,dx
+\chi C_3\int_\Omega (c-\chi)^2\,dx\right)
+\chi C_3\int_\Omega(c-\chi)^2\,dx\nonumber\\
&\le \chi C_3\int_\Omega(n-\chi)^2\,dx
-\frac{r}{\chi}\int_\Omega(n-\chi)^2\,dx,\qquad \forall~t>t_2.
\end{align}
Clearly, there is $\mu^*>\mu_*$ such that
$$\frac{r}{\chi}\ge\frac{r}{  \Big (\frac{r}{\mu^* }\Big)^{\frac{1}{\gamma}}   }
>\Big(\frac{r}{\mu^* }\Big)^{\frac{1}{\gamma}} C_3\ge \chi C_3,$$
whenever $\mu\geq\mu^*$. An integration of \eqref{ex25} w.r.t. the time variable gives
$$
\int_{t_2+1}^{\infty} \int_{\Omega}(n-\chi)^2 \,dxdt<\infty,
$$
which combined with the Schauder estimate of $n$ displayed in Theorem \ref{th3} guarantees that
$$
n(\cdot, t) \longrightarrow \chi \quad \text {in } ~L^{\infty}(\Omega)~ \text { as } ~t \longrightarrow \infty .
$$
Thus by recalling \eqref{ex21}, we can find $t_3>t_2+1$ ensuring
$$
\frac{1}{4 \chi}(n(\cdot, t)-\chi)^2 \leq H(n(\cdot, t)) \leq \frac{1}{\chi}(n(\cdot, t)-\chi)^2,\qquad \forall~t>t_3 .
$$
We let $C_4=\min\{1\, ,\, r-\chi^2 C_3 \}$ and infer from \eqref{ex25} that
$$
\frac{\mathrm{d}}{\mathrm{d} t}\left(\int_{\Omega}H\big(n(x,t)\big)\,dx
+\chi C_3\int_\Omega (c-\chi)^2\,dx\right)
+C_4\left(\int_{\Omega}H\big(n(x,t)\big)\,dx+\chi C_3\int_\Omega (c-\chi)^2\,dx\right) \leq 0,
\qquad \forall~t>t_3 .
$$
The Gr\"{o}nwall inequality yields
\begin{align*}
&\quad\int_{\Omega}H\big(n(x,t)\big)\,dx+\chi C_3\int_\Omega \big(c(x,t)-\chi\big)^2\,dx\\
&\leq \left(\int_{\Omega}H\big(n(x,t_3)\big)\,dx+\chi C_3\int_\Omega\big (c(x,t_3)-\chi\big)^2\,dx\right)
 e^{-C_4\left(t-t_3\right)}, \qquad\forall~ t>t_3 .
\end{align*}
In view of this and \eqref{ex21}, we obtain the decaying estimate
of $n(\cdot, t)-\chi$ in $L^2(\Omega)$ as below,
$$
\left\|n(\cdot, t)-\left(\frac{r}{\mu}\right)^{\frac{1}{\gamma}}\right\|_{L^2(\Omega)}^2+\left\|c(\cdot, t)-\left(\frac{r}{\mu}\right)^{\frac{1}{\gamma}}\right\|_{L^2(\Omega)}^2
\leq C_5 e^{-C_4\left(t-t_3\right)}, \qquad \forall~t>t_3
$$
with $C_5>0$.
Using the arguments as leading to \eqref{ex6}, we finally conclude
$$(n,c,\mathbf{u})\rightarrow\left(\left(\frac{r}{\mu}\right)^{\frac{1}{\gamma}}, \left(\frac{r}{\mu}\right)^{\frac{1}{\gamma}},0\right)\quad \text{in } L^\infty(\Omega)$$
 exponentially as $t\rightarrow\infty$, as intended.
\end{proof}
 \subsection{Higher regularity of solutions}

The H\"{o}lder continuity of solutions, typically serving as the beginning in the investigation of smooth regularities, 
stands as one of the most important directions for the study of PDEs.
It not only aids in improving the regularity of solutions but also offers compactness for certain problems.

 In the context of chemotaxis models, the H\"{o}lder continuity w.r.t the spatial variable can be utilized to pursue the asymptotic behaviors of solutions pointwise. Specifically, by verifying the validity of the embedding $C^{\gamma'}(\overline{\Omega})\hookrightarrow W^{\gamma,\ell}(\Omega)$ for all $\ell\in(1,\infty)$ and every pair of $\gamma,\gamma'\in(0,1)$ with $\gamma'>\gamma$,
 we then employ the Gagliardo-Nirenberg inequality to show that for any $\ell>\frac{N}{\gamma}$ and $q>1$,
\begin{align*}
\|n(\cdot,t)\|_{L^\infty(\Omega)}
&\le C_1\left(\|n(\cdot,t)\|^{\frac{\frac{1}{q}}{\frac{1}{q}+\frac{\gamma}{N}-\frac{1}{\ell}}}_{W^{\gamma,\ell}(\Omega)}
\|n(\cdot,t)\|^{   \frac{\frac{\gamma}{N}-\frac{1}{\ell}}{\frac{1}{q}+\frac{\gamma}{N}-\frac{1}{\ell}}  }_{L^q(\Omega)}
+\|n(\cdot,t)\|_{L^q(\Omega)}\right)\\
&\le C_2\left(\|n(\cdot,t)\|_{C^{\gamma}(\overline{\Omega})}
^{\frac{\frac{1}{q}}{\frac{1}{q}+\frac{\gamma}{N}-\frac{1}{\ell}}}
\|n(\cdot,t)\|_{L^q(\Omega)}
^{   \frac{\frac{\gamma}{N}-\frac{1}{\ell}}{\frac{1}{q}+\frac{\gamma}{N}-\frac{1}{\ell}}  }
+\|n(\cdot,t)\|_{L^q(\Omega)}\right)
\end{align*}
with $C_1=C_1(\ell,q,\gamma,N,\Omega)>0$ and $C_2=C_2(\ell,q,\gamma,\gamma',N,\Omega)>0$. 
An application of this inequality readily transforms the convergence in the $L^q$-norm into that in the $L^\infty$-norm.
\\[2pt]
\noindent{\bf Example (C)}
Our third example, derived from modeling the invasion movements of cancer cells,
describes the phenomena of chemotaxis and haptotaxis:
\begin{align}\label{example3}
\begin{cases}
n_t=\nabla \cdot\left(|\nabla n|^{p-2} \nabla n\right)- \nabla \cdot(\chi n \nabla c)-\nabla \cdot(\xi n \nabla w)+\mu n(1-n-w), &  x \in \Omega, t>0, \\ c_t=\Delta c-c+n, & x \in \Omega, t>0,  \\
w_t=-c w, & x \in \Omega, t>0
\end{cases}
\end{align}
endowed on the boundary condition $\left(|\nabla n|^{p-2} \nabla n-\chi n \nabla c-\xi n \nabla w\right) \cdot \nu=\partial_\nu c
=0$ on $\partial \Omega \times(0, \infty)$.
The non-negative initial data $\left(n_0, c_0, w_0\right)$ fulfill
\begin{align}\label{3initial}
\left\{\begin{array}{l}
n_0 \in L^{\infty}(\Omega),\left|\nabla n_0\right|^{p-2} \nabla n_0 \in L^2(\Omega), c_0, w_0 \in W^{2, \infty}(\Omega), \\
\partial_\nu\big(|\nabla n_0|^{p-2} \nabla n_0\big)\Big|_{\partial \Omega}
=\partial_\nu c_0\big|_{\partial \Omega}=\partial_\nu w_0\big|_{\partial \Omega}=0 .
\end{array}\right.
\end{align}
The next proposition states that the system \eqref{example3} admits a global H\"{o}lder-continuous solution.
\begin{proposition}\label{pro3}
 Let $\Omega \subseteq \mathbb{R}^N$ $(N \geq 2)$ be a bounded domain with a smooth boundary.
 Suppose $\chi>0$, $\xi>0$ and $\mu>0$.
 Assume that
 $$
p>1+\frac{N(2 N+3 \lambda)}{(N+1)(N+2 \lambda)} \quad \text { with  } \lambda=\frac{2 N}{(N-2)_{+}} .
$$
Then for any non-negative initial data $\left(n_0, c_0, w_0\right)$ satisfying \eqref{3initial},
there exist functions
$$
\left\{\begin{array}{l}
n \in L^{\infty}(\Omega \times(0, \infty))\cap C^{\gamma,\frac{\gamma}{p}}(\overline{\Omega}\times(0,\infty))
\text{ with some } \gamma\in(0,1), \\
c \in  C^{0}\left([0, \infty) ; W^{1,\infty}(\Omega)\right) \cap C^0(\overline{\Omega} \times[0, \infty))
\cap C^{2+\sigma,1+\frac{\sigma }{2} }(\overline{\Omega}  \times(0, \infty)), \\
w \in  C^{0}\big([0, \infty) ; W^{1,\infty}(\Omega) \big)
\end{array}\right.
$$
such that the triple $(n, c, w)$ forms a global solution of \eqref{example3} in the sense
that $n$ weakly satisfies \eqref{example3}$_1$,
and $c$ as well as $w$ classically satisfies \eqref{example3}$_2$ and \eqref{example3}$_3$.
\end{proposition}
\begin{proof}
[\bf Proof]
The weak and globally bounded solutions have been found in \cite[Theorem 1]{WCZ2023}.
Additionally, to determine the solvability by compactness arguments, 
the authors \cite{WCZ2023} replaced $\nabla \cdot\left(|\nabla n|^{p-2} \nabla n\right)$
with $\nabla \cdot\left((|\nabla n|+\varepsilon)^{p-2} \nabla n\right)$ in \eqref{example3}$_1$,
and then introduced a regularized system
\begin{align}\label{app}
\begin{cases}
n_t=\nabla \cdot\left((|\nabla n|+\varepsilon)^{p-2} \nabla n\right)- \nabla \cdot(\chi n \nabla c)-\nabla \cdot(\xi n \nabla w)+\mu n(1-n-w), &  x \in \Omega, t>0, \\ c_t=\Delta c-c+n, & x \in \Omega, t>0,  \\
w_t=-c w, & x \in \Omega, t>0
\end{cases}
\end{align}
 which possesses a global and classical solution $(n_\varepsilon,c_\varepsilon,w_\varepsilon) $ for every $\varepsilon\in(0,1)$.
According to \cite[Lemmas 8$\&$11]{WCZ2023}, for any $T>0$,
there exists an $\varepsilon$-independent constant $C_1$ such that
\begin{align}\label{3ex1}
\left\|n_{\varepsilon}\right\|_{L^{\infty}(\Omega)}
+\left\|c_{\varepsilon}\right\|_{W^{1, \infty}(\Omega)}+\left\|w_{\varepsilon}\right\|_{W^{1, \infty}(\Omega)} \leq C_1, \qquad \forall~t>0.
\end{align}
In light of Theorem \ref{th3}, we have that for any $t_0\in(0,T)$,
\begin{align}\label{3ex2}
|n_\varepsilon(x, t)  - n_\varepsilon(\hat{x}, \hat{t})|
\leq C_2\Big( |x-\hat{x}|+|t-\hat{t}|^{\frac{1}{p}} \Big)^{\gamma},
\qquad\forall~ (x,t),~(\hat{x},\hat{t})\in \overline{\Omega}\times[t_0,T),
\end{align}
where $C_2>0$ and $\gamma\in(0,1)$ are unform-in-$\varepsilon$ due to \eqref{3ex1}.
With the aid of the Arzel\`{a}–Ascoli theorem, we thus conclude the limit function $n$ of a subsequence of $\{n_\varepsilon\}_{\varepsilon\in(0,1)}$, as specified in
\cite[page 16]{WCZ2023},
belongs to the space $ C^{0}([t_0,T];L^{2}(\Omega))$.

Finally, let us set $f(s,x,t)=\mu s(1-s-w(x,t))$, $\hat{c}(x,t)=\chi c-\xi w $ for $(x,t)\in\overline{\Omega}\times(0,\infty)$.
Under this transformation, \cite[Theorem 1]{WCZ2023} tells us that $n$ weakly satisfies $$n_t=\nabla \cdot\left(|\nabla n|^{p-2} \nabla n\right)- \nabla \cdot( n \nabla \hat{c})+f(n,x,t),\qquad \forall~x\in\Omega,~ t>0$$
with $\nabla \hat{c}\in L^\infty(\Omega\times[t_0,T))$.
Hence, the expected H\"{o}lder continuity result immediately follows by an application of Theorem \ref{th3}.
\end{proof}
\noindent{\bf Example (D)}
Finally, let us consider the following chemotaxis with $p-$Laplacian involved:
\begin{align}\label{example4}
\begin{cases}n_t+\mathbf{u} \cdot \nabla n=\nabla \cdot\left(|\nabla n|^{p-2} \nabla n\right)-\nabla \cdot(n \nabla c), & x \in \Omega, t>0, \\
 c_t+\mathbf{u} \cdot \nabla c=\Delta c-n c, & x \in \Omega, t>0, \\
 \mathbf{u}_t=\Delta \mathbf{u}+\nabla P+n \nabla \Phi, & x \in \Omega, t>0, \\
 \nabla \cdot \mathbf{u}=0, & x \in \Omega, t>0,\end{cases}
\end{align}
where $n$ and $c$ satisfy homogeneous boundary conditions of Neumann type,
and $\mathbf{u}$ satisfies the homogeneous boundary condition of Dirichlet type.
The initial data fulfill
$$
\left\{\begin{array}{l}
n_0 \in C^\omega(\overline{\Omega}) ~(0<\omega<1), ~n_0 \geq 0, ~n_0 \not \equiv 0 \quad \text {on } \overline{\Omega}, \\
c_0 \in W^{1, \infty}(\Omega), ~c_0 \geq 0, ~ \sqrt{c_0} \in W^{1,2}(\Omega), \\
\mathbf{u}_0 \in D\left(\mathcal{A}_2^\gamma\right) \quad \text {for some } \gamma \in(\frac{3}{4}, 1).
\end{array}\right.
$$
\begin{proposition}\label{pro4}
Let $\Omega\subseteq \Bbb{R}^3$ be a bounded domain with smooth boundary, $\Phi \in W^{2, \infty}(\Omega)$ and $p>\frac{23}{11}$.
Then there exist functions
$$
\left\{\begin{array}{l}
n \in L^{\infty}(\Omega \times(0, \infty))\cap C^{\gamma,\frac{\gamma}{p}}(\overline{\Omega}\times(0,\infty))
\text{ with some } \gamma\in(0,1), \\
c \in  C^{0}\left([0, \infty) ; W^{1,\infty}(\Omega)\right) \cap C^0(\overline{\Omega} \times[0, \infty))
\cap C^{ 2+\sigma,1+\frac{\sigma }{2} }(\overline{\Omega}  \times(0, \infty)), \\
\mathbf{u }\in L^{\infty}\left(\Omega \times(0, \infty) ; \mathbb{R}^3\right)
\cap L_{l o c}^2\big([0, \infty) ;W_0^{1,2}(\Omega ; \mathbb{R}^3) \cap L_\sigma^2(\Omega ; \mathbb{R}^3)\big)
\cap C^0\left(\overline{\Omega}  \times[0, \infty) ; \mathbb{R}^3\right)
\end{array}\right.
$$
such that the triple $(n, c, \mathbf{u})$ forms a global solution of \eqref{example4} in the sense
that $c$ classically satisfies \eqref{example4}$_2$,
and $n$ together with $\mathbf{u}$ weakly satisfies \eqref{example4}$_1$ and \eqref{example4}$_3$.
\end{proposition}
\begin{proof}
[\bf Proof]
The global existence of bounded and weak solutions can be found in \cite{TL2020}.
Similar reasonings performed in Proposition \ref{pro3} can also
be carried out for the model \eqref{example4}, resulting in improved regularities
of solutions. To avoid repeated arguments, we omit details here.
\end{proof}

\section{Appendix}

We use the coming Lemma \ref{Asemigroup} and Lemma \ref{semigroup} to prepare some semigroup estimates, which 
can be found in \cite{CL}.

Define $D(\mathcal{A}_q):=W^{2,q}\left(\Omega ; \mathbb{R}^2\right) \cap W_0^{1,q}\left(\Omega ; \mathbb{R}^2\right) \cap L_\sigma^q(\Omega)$ with $q\ge1$. We let $\mathcal{A}:=-\mathcal{P} \Delta$ denote the realization of the Stokes operator on $D(\mathcal{A}_q)$. Therein, $\mathcal{P}$ stands for the Helmholtz projection from $L^q\left(\Omega ; \mathbb{R}^2\right)$ to $L_\sigma^q(\Omega)$.
As a sectorial and positive self-adjoint operator, $\mathcal{A}_q$ possesses closed fractional powers $\mathcal{A}_q^\theta$ defined on $D\left(\mathcal{A}_q^\theta\right)$ with $\theta\in \mathbb{R}$, where the norm is given by $\|\cdot\|_{D\left(\mathcal{A}_q^\theta\right)}:=\left\|\mathcal{A}_q^\theta(\cdot)\right\|_{L^q(\Omega)}$.
\begin{lemma}\label{Asemigroup}
Let $\Omega\subseteq\mathbb{R}^2$ be a bounded domain with smooth boundary, and let
 $\left\{e^{-\mathcal{A} t}\right\}_{t \geq 0}$ be the semigroup generated by the Stokes-operator $\mathcal{A}$.
Then there exists $\lambda_1>0$ with the following property:
For each $\theta\ge0$ and any $1< \ell\leq q< \infty$
there is $C>0$ depending only on $\theta$, $q$, $\ell$ and $\Omega$ such that
$$
\left\|\mathcal{A}^{ \theta } e^{-t \mathcal{A}} \varphi\right\|_{L^q(\Omega)} \leq C t^{-\theta-(\frac{1}{\ell}-\frac{1}{q})}e^{-\lambda_1 t}\|\varphi\|_{L^\ell(\Omega)}, \qquad \forall~ t>0
$$
is valid for all $\varphi \in L_\sigma^\ell(\Omega)$.
\end{lemma}

\begin{lemma}\label{semigroup}
Assume that $\Omega\subseteq\mathbb{R}^N$ $(N\ge 2)$ is a bounded domain with a smooth boundary.
Let $\{e^{t\Delta}\}_{t\geq0}$ be the Neumann heat semigroup generated by $-\Delta$, and let $\lambda_2 > 0$ denote the first nonzero
eigenvalue of $-\Delta$ in $\Omega$ under Neumann boundary conditions. Then there exist constants
$C_1, C_2$ depending only on $\Omega$, $q$ and $\ell$ which have the following properties:
\begin{itemize}
\item [{\rm (i)}] If $1\leq \ell\leq q\leq \infty$, then
$$\| e^{t\Delta}\varphi\|_{L^q(\Omega)}\leq C_1(q,\ell)
\Big(1+t^{-\frac{N}{2}(\frac{1}{\ell}-\frac{1}{q})}\Big)
e^{-\lambda_2 t}\|\varphi\|_{L^{\ell}(\Omega)}, \qquad \forall~ t>0$$
is true for each $\varphi\in L^{\ell}(\Omega)$.
\item [{\rm (ii)}] \ If $1\leq \ell\leq q\leq \infty$, then
$$\|\nabla e^{t\Delta}\varphi\|_{L^q(\Omega)}\leq C_2(q,\ell)\Big(1+t^{-\frac{1}{2}-\frac{N}{2}(\frac{1}{\ell}-\frac{1}{q})}\Big)
e^{-\lambda_2 t}\|\varphi\|_{L^{\ell}(\Omega)}, \qquad \forall~ t>0$$
holds for any $\varphi\in L^{\ell}(\Omega)$.
\end{itemize}
\end{lemma}

\begin{lemma}\label{A1}
Let $\Omega\subseteq \Bbb{R}^N$ $(N\ge2)$ be a bounded domain with smooth boundary and $0<\sigma<2/N$.
For any $q>1$, there are constants $C=C(q,\sigma,\Omega,N)>0$ and $\kappa=\kappa(q,\sigma,\Omega,N)>0$ such that
for any solution $(n,c)$ of \eqref{example1}, we can find $\hat{t}>0$ satisfying
\begin{align*}
\|n(\cdot,t)\|_{L^{ q}(\Omega)} \leq CM^{\kappa}+C,\qquad \forall~t>\hat{t},
\end{align*}
where $M:=\dashint_{\Omega}n_0\,dx$.
\end{lemma}
\begin{proof}
[\bf Proof]
The existence of global and classical solutions emanating from initial data satisfying \eqref{1initial}
can be found in \cite{LT2016}.
An integration to \eqref{example1}$_1$ over $\Omega$ shows that
\begin{align}\label{M}
\dashint_{\Omega}n(\cdot,t)\,dx\equiv \dashint_{\Omega}n_0\,dx =M,\qquad \forall~t>0.
\end{align}
Since $0<\sigma< 2/N$, then we can fix
$s\in \left[1,\frac{N}{(N\sigma-1)_+}\right)$ such that
\begin{align}\label{s}
\sigma-\frac{1}{N} <\frac{1}{s}<\frac{1}{N}.
\end{align}
Thus, based on the variation-of-constant formula, there holds that
\begin{align}\label{A11}
\|\nabla c(\cdot, t)\|_{L^{ s }(\Omega)}
\leq &
\big\|\nabla e^{t(\Delta-1)} c_{0}\big\|_{L^{s }(\Omega)}
+\int_{0}^{t}\big\|\nabla e^{-(t-s)(\Delta-1) } n^{\sigma} \big\|_{L^{ s }(\Omega)}\,ds,\qquad\forall~ t>0.
\end{align}
 We employ Lemma \ref{semigroup} (i) to get
\begin{align}\label{A12}
\big\|\nabla e^{t(\Delta-1) } c_0\big\|_{L^{ s}(\Omega)}
 \le  C_1 e^{-\lambda_2 t}\big(1+t^{-\frac{1}{2}} \big)
\left\|\nabla c_0\right\|_{L^{ s }(\Omega)} ,\qquad\forall~ t>0
\end{align}
 and
\begin{align}\label{A12'}
&\quad\int_{0}^{t}\big\|\nabla e^{-(t-s)\mathcal{A}} n^{\sigma}(\cdot, s)\big\|_{L^{ s }(\Omega)} \,ds \nonumber \\
&\le
C_2\int_{0}^{t}
(t-s)^{-\frac{1 }{2}- \frac{N}{2}(\sigma-\frac{1}{s})  } e^{-\lambda_2(t-s)}
\big\|  n^{\sigma}(\cdot, s)\big \|_{L^{ \frac{1}{\sigma}  }(\Omega)}\,ds \nonumber \\
&\le C_3 M^{\sigma},\qquad\forall~ t>0
\end{align}
with positive constants $C_1$, $C_2$, $C_3$ depending only on $\sigma$, $\Omega$, $N$.
We pick a time point $t_1>1$ sufficiently large such that
$2e^{-\lambda_2 t_1}\|\nabla c_0\|_{L^{ s }(\Omega)} \le 1$, and then,
\begin{align}\label{A12''}
\|\nabla c(\cdot, t)\|_{L^{ s }(\Omega)}
\leq 2C_1+C_3 M^{\sigma}=:\Lambda_M,\qquad\forall~ t>t_1.
\end{align}

Now we let $\ell>\max\big\{ \frac{s}{2}+1\,,\,  \frac{(N-2)s}{N} \,,\,\frac{s N}{2}  \big\}$ big enough assure
\begin{align*}
\frac{2\sigma(N-2)}{ \frac{2\ell}{s} \cdot N }<1-\frac{(N-2)(\ell-1)}{N \ell}.
\end{align*}
With setting $q:=\frac{2\ell}{s} $, the selection of $\ell$ ensures that $q>N$ and
\begin{align}
\frac{2\sigma(N-2)}{ q \cdot N }<1-\frac{(N-2)(\ell-1)}{N \ell}.
\end{align}
This inequality allows us to find $\chi>1$ such that
\begin{align}\label{530a1}
\frac{2\sigma(N-2)}{Nq}<\frac{1}{\chi}<\min\left\{2\sigma\, ,\,1-\frac{(N-2)(\ell-1)}{N\ell}\right\}.
\end{align}
We also need to select $\nu>1$ fulfilling 
\begin{align}\label{530a2}
\max\left\{1-\frac{2}{N}\,,\, 1-\frac{2}{s}\right\}<\frac{1}{\nu}<1-\frac{N-2}{N\ell}.
\end{align}
Testing (\ref{example1})$_1$ by $n^{q-1}$ and using Young's inequality, one sees that for any $t>0$,
\begin{align} \label{A13}
\frac{1}{q}\frac{\mathrm{d}}{\mathrm{d}t}\int_{\Omega}n^q\,dx
+(q-1)\int_{\Omega}n^{q-2}|\nabla n|^2\,dx  \notag
=& -
(q-1)\int_{\Omega}n^{q-1}\nabla n\cdot \nabla c\,dx \notag \\
\leq& \frac{q-1}{2}\int_{\Omega}n^{q-2}|\nabla n|^2\,dx
+\frac{q-1}{2}\int_{\Omega}n^{q}|\nabla c|^2\,dx,
\end{align}
which yields that
\begin{align} \label{A14}
&\frac{\mathrm{d}}{\mathrm{d}t}\int_{\Omega}n^{q}\,d x
+\frac{2(q-1)}{q}\int_{\Omega}\big|\nabla n^{\frac{q}{2}}\big|^2\,d x
\leq q(q-1)\int_{\Omega}n^{q}|\nabla c|^2\,d x,\qquad\forall~ t>0.
\end{align}
By a straightforward computation, we verify that
$$\Delta |\nabla c|^2=2|D^2c|^2+2\nabla c\cdot\nabla \Delta c,\qquad \forall~x\in\Omega,~ t>0. $$
It follows by (\ref{example1})$_2$ and the pointwise estimate $|\Delta c|^2 \leq N |D^2c|^2$ in $\Omega\times(0,\infty)$
that
\begin{align} \label{A15}
(|\nabla c|^2)_t +\frac{2}{N}|\Delta c|^2 +2|\nabla c|^2
\leq
\Delta |\nabla c|^2 +2\nabla c \cdot \nabla n^{\sigma},\qquad \forall~x\in\Omega,~ t>0.
\end{align}
Multiplying (\ref{A15}) by $|\nabla c|^{2(\ell-1)}$ and integrating over $\Omega$ imply that
\begin{align}
&\frac{1}{\ell}\frac{\mathrm{d}}{\mathrm{d}t}\int_{\Omega}|\nabla c|^{2\ell}\,d x
+\frac{2}{N}\int_{\Omega}|\nabla c|^{2(\ell-1)}|\Delta c|^2\,d x
+2\int_{\Omega}|\nabla c|^{2\ell}\,d x \notag\\
=& \int_{\Omega}|\nabla c|^{2(\ell-1)}\Delta |\nabla c|^2\,d x
+2\int_{\Omega}|\nabla c|^{2(\ell-1)}\nabla c\cdot \nabla n^{\sigma}\,d x,\qquad \forall~ t>0.
\end{align}
Through the integration by parts and Young's inequality, we have
\begin{align*}
 2\int_{\Omega}|\nabla c|^{2(\ell-1)}\nabla c\cdot \nabla n^{\sigma}\,d x
&=-2(\ell-1)\int_{\Omega}|\nabla c|^{2(\ell-2)}\nabla |\nabla c|^2 \cdot \nabla c n^{\sigma}\,d x
-2\int_{\Omega}|\nabla c|^{2(\ell-1)}\Delta c  n^{\sigma}\,d x\notag\\
&\leq\frac{(\ell-1)}{2}\int_{\Omega}|\nabla c|^{2(\ell-2)}\big|\nabla |\nabla c|^2\big|^2\,d x
+\frac{2}{N}\int_{\Omega}|\nabla c|^{2(\ell-1)}|\Delta c|^2\,d x \notag\\
&\quad+\Big(2(\ell-1)+\frac{N}{2}\Big)\int_{\Omega}|\nabla c|^{2(\ell-1)} n^{2\sigma}\,d x
\end{align*}
and
\begin{align*}
\int_{\Omega}|\nabla c|^{2(\ell-1)}\Delta |\nabla c|^2\,d x
&=-(\ell-1)\int_{\Omega}|\nabla c|^{2(\ell-2)}\big|\nabla |\nabla c|^2\big|^2\,d x
+ \int_{\partial \Omega}|\nabla c|^{2(\ell-1)}
\frac{\partial |\nabla c|^2}{\partial \nu}\,dS \notag\\
&\leq-\frac{2(\ell-1)}{\ell^2}\int_{\Omega}\big|\nabla |\nabla c|^\ell\big|^2\,d x
+C_4\int_{ \Omega}|\nabla c|^{2\ell}\,d x,
\end{align*}
where $C_4=C_4(\sigma,\Omega,N)>0$, we utilized \cite[Lemma 4.2]{MS14} and the trace inequality to govern $\int_{\partial \Omega}|\nabla c|^{2(\ell-1)}
\frac{\partial |\nabla c|^2}{\partial \nu}\,dS$.
Thereby, there holds that
\begin{align}\label{A16}
&\quad\frac{\mathrm{d}}{\mathrm{d}t}\int_{\Omega}|\nabla c|^{2\ell}\,d x
+\frac{\ell-1}{2\ell}\int_{\Omega}\big|\nabla |\nabla c|^\ell\big|^2\,d x \nonumber\\
&\leq \Big(2(\ell-1)+\frac{N}{2}\Big)\ell\int_{\Omega}|\nabla c|^{2(\ell-1)}n^{2\sigma}\,dx
+C_4\ell\int_{\Omega}|\nabla c|^{2\ell}\,dx,\qquad \forall~ t>0.
\end{align}
We merge (\ref{A16}) with (\ref{A14}) and use H\"{o}lder's inequality to find $C_5>0$ such that
the following inequality
\begin{align} \label{A17}
&\frac{\mathrm{d}}{\mathrm{d}t}\int_{\Omega}(n^q+|\nabla c|^{2\ell})\,dx
+\frac{2(q-1)}{q}\int_{\Omega}\big|\nabla n^{\frac{q}{2}}\big|^2 \,dx
+ \frac{\ell-1}{2\ell}\int_{\Omega}\big|\nabla |\nabla c|^{\ell}\big|^2 \,dx \notag \\
\leq& C_5\int_{\Omega}n^{q}|\nabla c|^{2}\,dx
+ C_5\int_{\Omega}n^{2\sigma}|\nabla c|^{2(\ell-1)}\,dx
+C_5\int_{\Omega}|\nabla c|^{2\ell}\,d x\notag \\
\leq&C_5\left(\int_{\Omega}n^{q\nu}\,dx\right)^{\frac{1}{\nu}}
\left(\int_{\Omega}|\nabla c|^{2\nu'}\,dx\right)^{\frac{1}{\nu'}}
+C_5\left(\int_{\Omega}n^{2\sigma \chi}\,dx\right)^{\frac{1}{\chi}}
\left(\int_{\Omega}|\nabla c|^{2(\ell-1)\chi'}\,dx\right)^{\frac{1}{\chi'}}\notag\\
&+C_5\int_{\Omega}|\nabla c|^{2\ell}\,dx,\qquad \forall~ t>0
\end{align}
holds for $\nu$ selected in \eqref{530a1} and $\nu'=\frac{\nu}{\nu-1}$.
The Gagliardo-Nirenberg inequality infers
\begin{align}\label{A18}
\big\|n^{\frac{q}{2}}\big\|_{L^{2\nu}(\Omega)}^{ 2  }
&\leq C_{6}\big\|\nabla n^{\frac{q}{2}}\big\|_{L^2(\Omega)}^{ 2\theta }
\big\|n^{\frac{q}{2}}\big\|_{L^{\frac{2}{q}}(\Omega)}^{ 2(1-\theta)  } \notag
+ C_{6}\big\|n^{\frac{q}{2}}\big\|_{L^{\frac{2}{q}}(\Omega)}^{2} \notag \\
&\leq C_6\left(\int_{\Omega}\big|\nabla n^{\frac{q}{2}}\big|^2\,dx\right)^{\theta}M^{q(1-\theta)} + C_6M^q,
\end{align}
where $C_6=C_6(\sigma,\Omega,N)>0$ and
$$\theta=\frac{   \frac{q}{2}-\frac{1}{2\nu}  }{  \frac{1}{N}-\frac{1}{2}+\frac{q}{2}  }
\in(0,1).$$
Given (\ref{A12''}), we apply the Gagliardo-Nirenberg inequality once again to derive that
for all $t>t_1$,
\begin{align} \label{A19}
\big\||\nabla c|^{\ell}\big\|_{L^{\frac{2\nu'}{\ell}}(\Omega)}
^{\frac{2}{\ell}}
&\leq C_{7}\big\|\nabla |\nabla c|^\ell\big\|_{L^2(\Omega)}^{\frac{2\delta}{\ell}}
\big\||\nabla c|^{\ell}\big\|_{L^{\frac{s}{\ell}}(\Omega)}^{\frac{2(1-\delta)}{\ell}}
+ C_{7}\big\||\nabla c|^\ell\big\|_{L^{\frac{s}{\ell}}(\Omega)}^{\frac{2}{\ell}} \notag\\
&\leq C_7\left(\int_{\Omega}\big|\nabla |\nabla c|^\ell\big|^2\,dx\right)^{\frac{\delta}{\ell}} 
\Lambda_M^{2(1-\delta)}
 + C_7\Lambda_M^{2},
\end{align}
where $C_7=C_7(\sigma,\Omega,N)>0$ and
$$\delta = \frac{\frac{\ell}{s}-\frac{\ell}{2\nu'}  }
{\frac{1}{N}-\frac{1}{2}+\frac{\ell}{s}} \in (0,1).$$
Similarly, there exist $C_8>0,C_9>0$ depending $\sigma$, $\Omega$ and $N$ satisfying that
for any $t>t_1$,
\begin{align}\label{A20}
\big\|n^{\frac{q}{2}}\big\|_{L^{\frac{4\sigma \chi}{q}}(\Omega)}^{\frac{4\sigma}{q}}
\leq C_8\left(\int_{\Omega}|\nabla n^{\frac{q}{2}}|^2\,dx\right)^{\frac{2\sigma\bar{\theta}}{q}}
M^{2\sigma(1-\bar{\theta})}
+C_8 M^{2\sigma}
\end{align}
and
\begin{align}\label{A21}
\big\||\nabla c|^\ell\big\|_{L^{\frac{2(\ell-1)\chi'}{\ell}}(\Omega)}
^{\frac{2(\ell-1)}{\ell}}
\leq C_9\left(\int_{\Omega}\big|\nabla |\nabla c|^\ell\big|^2\,dx\right)^{\frac{(\ell-1)\bar{\delta}}{\ell}}
\Lambda_M^{2(1-\ell)(1-\bar{\delta})}
+C_9\Lambda_M^{2(1-\ell)}
\end{align}
with $\chi$ specified in \eqref{530a2}, $\chi'=\frac{\chi}{\chi-1}$
and 
\begin{align*}
\bar{\theta}=\frac{ \frac{q}{2}-\frac{q}{4\sigma \chi}  }{ \frac{1}{N}-\frac{1}{2}+\frac{q}{2}  } \in (0,1)\quad \text{as well as}\quad
\bar{\delta} = \frac{    \frac{\ell}{s}-\frac{\ell}{2(\ell-1)\chi'}   }{ \frac{1}{N}-\frac{1}{2}+\frac{\ell}{s}  }
\in (0,1).
\end{align*}
Inserting (\ref{A18})--(\ref{A21}) into (\ref{A17}) and applying Young's inequality result in the estimate
\begin{align} \label{A21'}
&\frac{\mathrm{d}}{\mathrm{d}t}\int_{\Omega}(n^q+|\nabla c|^{2\ell})\,dx
+\frac{2(q-1)}{q}\int_{\Omega}|\nabla n^{\frac{q}{2}}|^2\,dx
+ \frac{\ell-1}{2\ell}\int_{\Omega}\big|\nabla |\nabla c|^{\ell}\big|^2\,dx \notag \\
\leq& C_{10}M^{q(1-\theta)}\Lambda_M^{2(1-\delta)}
\left(\int_{\Omega}\big|\nabla n^{\frac{q}{2}}\big|^2\,dx\right)^{\theta}
\left(\int_{\Omega}\big|\nabla |\nabla c|^\ell\big|^2\,dx\right)^{ \frac{\delta}{\ell} }
 \notag \\
 &+ C_{10}M^{2\sigma(1-\bar{\theta})}\Lambda_M^{2(1-\ell)(1-\bar{\delta})}
 \left(\int_{\Omega}\big|\nabla n^{\frac{q}{2}}\big|^2\,dx\right)^{\frac{2\sigma\bar{\theta}}{q}}
\left(\int_{\Omega}\big|\nabla |\nabla c|^\ell\big|^2\,dx\right)^{\frac{(\ell-1)\bar{\delta}}{\ell}}\notag\\
&+C_{10}\int_{\Omega}|\nabla c|^{2\ell}\,dx
+C_{10}M^{\kappa_1}
+C_{10},\qquad \forall~ t>t_1
\end{align}
with $\kappa_1=\kappa_1(\sigma,\Omega,N)>0$ and $C_{10}=C_{10}(\sigma,\Omega,N)>0$,
where we can verify that
\begin{align} \label{small}
\theta+\frac{\delta}{\ell}=\frac{   \frac{q}{2}-\frac{1}{2\nu}  }{  \frac{1}{N}-\frac{1}{2}+\frac{q}{2}  }
+    \frac{\frac{1}{s}-\frac{1}{2\nu'}}{\frac{1}{N}-\frac{1}{2}+\frac{\ell}{s}}
=\frac{\frac{1}{s}+ \frac{q}{2}-\frac{1}{2}}{\frac{1}{N}+\frac{q}{2}-\frac{1}{2}}<1
\end{align}
because of $s<\frac{1}{N}$,
and
\begin{align}\label{small}
\frac{2\sigma\bar{\theta}}{q}+\frac{(\ell-1)\bar{\delta}}{\ell}
=\frac{\sigma-\frac{1}{2\chi}  }{ \frac{1}{N}-\frac{1}{2}+\frac{q}{2}  }
+\frac{    \frac{\ell-1}{s}-\frac{1}{2\chi'}   }{ \frac{1}{N}-\frac{1}{2}+\frac{\ell}{s}  }
=\frac{   \sigma+ \frac{\ell-1}{s}-\frac{1}{2}   }{ \frac{1}{N}-\frac{1}{2}+\frac{\ell}{s}  }<1
\end{align}
due to $s>\sigma-\frac{1}{N}$.
Young's inequality applied on (\ref{A21'}) enables us to find $C_{11}>0$ and $\kappa_2=\kappa_2(\sigma,\Omega,N)>0$ such that
\begin{align}\label{A22}
&\quad\frac{\mathrm{d}}{\mathrm{d}t}\int_{\Omega}(n^q+|\nabla c|^{2\ell})\,dx
+\frac{q-1}{q}\int_{\Omega}\big|\nabla n^{\frac{q}{2}}\big|^2\,dx
+ \frac{\ell-1}{4\ell}\int_{\Omega}\big|\nabla |\nabla c|^{\ell}\big|^2\,dx\notag \\
& \leq C_{10}\int_{\Omega}|\nabla c|^{2\ell}\,dx +C_{11}M^{\kappa_2}+ C_{11},\qquad \forall~ t>t_1.
\end{align}
The Gagliardo-Nirenberg inequality tells us that
\begin{align} \label{A23}
\big\|n^{\frac{   q}{2}}\big\|_{L^{2}(\Omega)}^{2}
\leq C_{12}\left(\int_{\Omega}
\big|\nabla n^{\frac{q}{2}}\big|^2\,dx\right)^{\frac{\frac{q}{2}-\frac{1}{2}}{\frac{1}{N}-\frac{1}{2}+\frac{q}{2}}}
M^{\frac{\frac{q}{N}}{\frac{1}{N}-\frac{1}{2}+\frac{q}{2}}}
+C_{12}M^q
\end{align}
and
\begin{align} \label{A24}
(C_{10}+1)\big\||\nabla c|^\ell\big\|_{L^2(\Omega)}^2\leq C_{13}\left(\int_{\Omega}\big|\nabla |\nabla c|^\ell\big|^2\,dx\right)^{\frac{\frac{\ell}{s}-\frac{1}{2}}{\frac{1}{N}-\frac{1}{2}+\frac{\ell}{s}}}
\Lambda_M^{  \frac{ \frac{2\ell}{N}  }{   \frac{s}{N} -\frac{s}{2}+\ell}  }+C_{13}\Lambda_M^{  \frac{2\ell}{s} }
\end{align}
with $C_{12},C_{13}>0$.
It can be deduced from (\ref{A22})--(\ref{A24}) that there are $\kappa_3$, $C_{14}$ and $C_{15}$ such that
\begin{align*}
\frac{\mathrm{d}}{\mathrm{d}t}\int_{\Omega}(n^q+|\nabla c|^{2\ell})\,dx
+C_{14}\left(\int_{\Omega}\big|\nabla n^{\frac{q}{2}}\big|^2dx
+ \int_{\Omega}\big|\nabla |\nabla c|^{\ell}\big|^2\mathrm{d}x\right)
\leq C_{15}M^{\kappa_3}+C_{15},\qquad \forall~t>t_1.
\end{align*}
An argument of ODI to the above inequality guarantees the existence of $t_2>t_1$ fulfilling
\begin{align*}
\int_{\Omega}n^{q}(x,t)\,dx \leq C_{15}M^{\kappa_3}+2C_{15},\qquad \forall~t>t_2.
\end{align*}
The proof is finished.
\end{proof}

\begin{lemma}\label{A2}
Let $\Omega\subseteq \Bbb{R}^2$ be a bounded domain with smooth boundary and $\gamma>1$.
Assume that $D$ and $S$ satisfy \eqref{cDS} with $\alpha\ge0$ and $\beta\leq\gamma$.
Then there are constants $C>0$ and $\mu_*>0$ depending only on $a_0,b_0,\alpha,\beta,r$ such that whenever $\mu>\mu_*$, then for any classical and globally bounded solutions $(n,c,\mathbf{u})$ to \eqref{example2} there exists a time point
$\hat{t}\in(0,\infty)$ fulfilling
\begin{align*}
\|n(\cdot,t)\|_{L^{ 2}(\Omega)} \leq C \quad \text{and} \quad
\|\mathbf{u}(\cdot,t)\|_{L^{ \infty}(\Omega)} \leq C,\qquad \forall~t>\hat{t}.
\end{align*}
\end{lemma}

\begin{proof}
[\bf Proof]
The Poincar\'{e} inequality and Sobolev estimates tell that
\begin{align}\label{2Ap}
\int_{\Omega} |\nabla \psi|^2 \,dx\geq C_p \int_{\Omega}  \psi^2\,dx\qquad {\rm for~any~}\psi \in W^{1,2}_0(\Omega)
  \end{align}
and
\begin{align}\label{2As}
\int_{\Omega} |\Delta \psi|^2\,dx+\int_{\Omega} \psi^2\,dx
 \geq C_s \int_{\Omega}  |\nabla \psi|^2 \,dx
\qquad {\rm for~every~}\psi\in W^{2,2}(\Omega)\cap W_0^{1,2}(\Omega).
  \end{align}
We thus define a constant $\bar{C}:=\max\left\{\frac{C_p}{4}\,,\,\frac{ C_s}{2}\,,\, \lambda_1 \right\}$
with $\lambda_1>0$ specified as in Lemma \ref{Asemigroup}.

The equation \eqref{example2}$_1$ ensures that
  \begin{align}\label{2A1}
   \frac{\mathrm{d}}{\mathrm{d}t}\int_{\Omega} n\,dx= r \int_{\Omega} n \,dx-\mu \int_{\Omega} n^{1+\gamma} \,dx
  \leq r \int_{\Omega} n \,dx-\mu |\Omega|^{-\gamma}\left(\int_{\Omega} n \,dx \right)^{1+\gamma},
  \qquad\forall~ t>0.
  \end{align}
  According to the Bernoulli inequality \cite[Lemma 1.2.4]{CD}, there holds that
  \begin{align}\label{1}
\limsup_{t\rightarrow\infty} \int_{\Omega} n(\cdot,t) \,dx\le |\Omega|\Big(\frac{r}{\mu}\Big)^{\frac{1}{\gamma}}.
\end{align}
This gives a time point $t_1>0$ such that $\int_{\Omega} n(\cdot,t)\,dx<2|\Omega|\big(\frac{r}{\mu}\big)^{\frac{1}{\gamma}}$ for any $t>t_1$.
This along with \eqref{2A1} shows that
 \begin{align*}
   \frac{\mathrm{d}}{\mathrm{d}t}\left( e ^{\bar{C} t}\int_{\Omega} n\,d x\right)
   &=(r+\bar{C})e ^{\bar{C} t} \int_{\Omega} n\,dx-\mu e ^{\bar{C} t} \int_{\Omega} n^{1+\gamma}\,dx\\
   &\leq 2|\Omega|(r+\bar{C})e ^{\bar{C} t}\Big(\frac{r}{\mu}\Big)^{\frac{1}{\gamma}}
   -\mu e ^{\bar{C} t} \int_{\Omega} n^{1+\gamma}\,dx,\qquad \forall~t>t_1.
  \end{align*}
Upon integration over $(t_1,t)$, the above ODI yields that
\begin{align}\label{2A1'}
  \mu \int_{t_1}^t  e ^{\bar{C} (s-t)} \int_{\Omega} n^{1+\gamma} \,dxds
  \le 4|\Omega|\frac{\bar{C}+r}{\bar{C}}\Big(\frac{r}{\mu}\Big)^{\frac{1}{\gamma}},\qquad \forall~t>t_1,
  \end{align}
which immediately implies that
\begin{align}\label{n2}
    \int_{t_1}^t  e ^{\bar{C} (s-t)} \int_{\Omega} n^{2} \,dxds
  & \le  \int_{t_1}^t  e ^{\bar{C} (s-t)} \int_{\Omega} n^{1+\gamma} \,dxds+\frac{1}{\bar{C}}\nonumber\\
  &\le  4|\Omega|\frac{\bar{C}+r}{\bar{C}\mu}\Big(\frac{r}{\mu}\Big)^{\frac{1}{\gamma}}
  +\frac{1}{\bar{C}},\qquad \forall~t>t_1.
  \end{align}
From \eqref{example2}$_3$ and Young's inequality, we obtain
\begin{align*}
   \frac{1}{2}\frac{\mathrm{d}}{\mathrm{d}t}\int_{\Omega} \mathbf{u}^2\,dx+\int_{\Omega} |\nabla \mathbf{u}|^2\,dx
   \le  \int_{\Omega} n \mathbf{u}\,dx\le  \frac{1}{2 C_p}\int_{\Omega} n^2 \,dx
   +  \frac{C_p}{2}\int_{\Omega}  \mathbf{u}^2\,dx,\qquad \forall~t>t_1.
  \end{align*}
Thus, it follows by \eqref{2Ap} that
     \begin{align*}
     \frac{\mathrm{d}}{\mathrm{d}t}\left( e ^{   \frac{C_p t}{4}  }\int_{\Omega} \mathbf{u}^2\,dx\right)
     +\frac{ e^{   \frac{C_p t}{4}  } }{  4  }\int_{\Omega} |\nabla \mathbf{u}|^2 \,dx
   \le \frac{ e^{   \frac{C_p t}{4}  }  }{C_p} \int_{\Omega} n^2 \,dx,\qquad \forall~t>t_1.
  \end{align*}
In view of \eqref{n2}, one can find $t_2>t_1$ such that
\begin{align}\label{2A4}
\|\mathbf{u}(\cdot,t)\|^2_{ L^2(\Omega)}\le \frac{ C_1}{\mu^{\frac{1}{\gamma}+1}},
  \qquad  \int_{t_1}^t  e ^{ \frac{s-t}{2} } \int_{\Omega} |\nabla \mathbf{u}|^2\,dxds
  \le\frac{ C_1}{\mu^{\frac{1}{\gamma}+1}},\qquad \forall~t>t_2.
  \end{align}
with $C_1=C_1(r,\gamma,\Omega)>0$.
  Now we test the equation \eqref{example2}$_3$ with $\Delta \mathbf{u}$ and integrate by parts to get
  \begin{align}\label{2A5}
   \frac{1}{2}\frac{\mathrm{d}}{\mathrm{d}t}\int_{\Omega} |\nabla \mathbf{u}|^2\,d x+\int_{\Omega} |\Delta\mathbf{u}|^2\,dx
   &\le  \int_{\Omega} n \Delta\mathbf{u}\,dx+ \int_{\Omega} \mathbf{u}\cdot \nabla \mathbf{u} \Delta\mathbf{u}\,dx \nonumber\\
   & \le  \int_{\Omega} n^2\,dx
   +  \frac{1}{2}\int_{\Omega}   |\Delta\mathbf{u}|^2\,dx
   +\int_{\Omega}   |\mathbf{u}\cdot \nabla \mathbf{u}|^2\,dx,\qquad\forall~ t>0,
  \end{align}
  where H\"{o}lder's inequality, the Gagliardo-Nirenberg inequality and Young's inequality entail that
  \begin{align}\label{2A6}
  \int_{\Omega}   |\mathbf{u}\cdot \nabla \mathbf{u}|^2 \,dx
  \le \|\mathbf{u}\|^2_{L^4(\Omega)} \|\nabla\mathbf{u}\|_{L^4(\Omega)} ^2
&  \le C_2\Big( \|\nabla\mathbf{u}\|_{L^2(\Omega)}^{\frac{1}{2}} \|\mathbf{u}\|_{L^2(\Omega)}^{\frac{1}{2}}  \Big)^2
  \Big(\|\Delta\mathbf{u}\|_{L^2(\Omega)}^{\frac{1}{2}} \|\nabla\mathbf{u}\|_{L^2(\Omega)}^{\frac{1}{2}}  \Big)^2\nonumber\\
&  \le  \frac{1}{4}\|\Delta\mathbf{u}\|^2_{L^2(\Omega)}+ C^2_2\|\mathbf{u}\|^2_{L^2(\Omega)}\|\nabla\mathbf{u}\|^4_{L^2(\Omega)}
  \end{align}
  and
\begin{align}\label{2A7}
\|\nabla\mathbf{u}\|^4_{L^2(\Omega)} \le C_3\Big(\|\Delta\mathbf{u}\|^2_{L^2(\Omega)}  \|\mathbf{u}\|^2_{L^2(\Omega)}
+\|\mathbf{u}\|^4_{L^2(\Omega)} \Big)
  \end{align}
  with $C_2,C_3>0$. It can be obtained by a combination of \eqref{2A5}--\eqref{2A7},
   Young's inequality as well as \eqref{2A4} that
  \begin{align*}
   \frac{\mathrm{d}}{\mathrm{d}t}\int_{\Omega} |\nabla \mathbf{u}|^2\,dx
   +\frac{1}{2}\int_{\Omega} |\Delta\mathbf{u}|^2\,dx
 &   \le  C_4\Big(\|\Delta\mathbf{u}\|^2_{L^2(\Omega)}  \|\mathbf{u}\|^4_{L^2(\Omega)}
+\|\mathbf{u}\|^6_{L^2(\Omega)} \Big)
+2\int_{\Omega} n^2 \,dx \\
&\le\frac{ C_4 C_1^2}{\mu^{\frac{2}{\gamma}+2}}\|\Delta\mathbf{u}\|^2_{L^2(\Omega)}
+\frac{ C_4 C_1^3}{\mu^{\frac{3}{\gamma}+3}}
+2\int_{\Omega} n^2\,dx,\qquad \forall~t>t_2
  \end{align*}
with $C_4=2C_2^2C_3$. We let $\mu_*>1$ be taken large enough satisfying
  $$\frac{ C_4 C_1^2}{\mu_*^{\frac{2}{\gamma}+2}}\le \frac{1}{4}.$$
   Therefore, the assumption $\mu>\mu_*$ and \eqref{2As} imply that
     \begin{align*}
   \frac{\mathrm{d}}{\mathrm{d}t}\left( e ^{ \frac{C_s t}{2}  }\int_{\Omega} | \nabla \mathbf{u}|^2 \,dx\right)
   \le  2e^{ \frac{C_s t}{2}} \left(\int_{\Omega} n^2\mathrm{d}x
+\frac{ C_4 C_1^3}{\mu^{\frac{3}{\gamma}+3}}
+\frac{ C_1}{\mu^{\frac{1}{\gamma}+1}}\right),\qquad \forall~t>t_2.
  \end{align*}
   Using an ODI argument, we infer from the above display and \eqref{n2} that there is $t_3>t_2$ ensuring
\begin{align}\label{2A8}
\|  \nabla \mathbf{u} (\cdot,t) \|^2_{ L^2(\Omega)}\le
2 C_4 C_1^3 r^{\frac{3}{\gamma}}+2C_1+
+8 r^{\frac{1}{\gamma}}|\Omega|\frac{\bar{C}+r}{\bar{C}\mu}
  +\frac{2}{\bar{C}},\qquad \forall~t>t_3.
\end{align}
Similar reasonings as leading to \eqref{2A4} and \eqref{2A8} allow us to derive
\begin{align}\label{2A8'}
\| c(\cdot,t) \|^2_{ W^{1,2}(\Omega)}\le C_5,\qquad \forall~t>t_4
\end{align}
  with some $t_4>t_3$ and $C_5=C_5(r,\gamma,\Omega)>0$.

In view of the variation-of-constants formula according to \eqref{example2}$_2$, there holds
\begin{align}\label{cwaneqk}
\big\|\mathcal{A}^{  \frac{1}{2} }  \mathbf{u}(\cdot, t)\big\|_{L^{1+\gamma}(\Omega)}
\leq &
\big\|\mathcal{A}^{  \frac{1}{2} }  e^{-(t-t_2)\mathcal{A}} \mathbf{u}(\cdot,t_3)\big\|_{L^{1+\gamma}(\Omega)}
+\int_{t_3}^{t} \big\|\mathcal{A}^{  \frac{1}{2} }  e^{-(t-s)\mathcal{A}} \mathbf{u}\cdot\nabla \mathbf{u}(\cdot, s) \big\|_{L^{1+\gamma}(\Omega)} \,ds
\nonumber\\
&+\int_{t_3}^{t}\big\|\mathcal{A}^{  \frac{1}{2} }  e^{-(t-s)\mathcal{A}} n \big\|_{L^{1+\gamma }(\Omega)}\,ds,\qquad\forall~t >t_3.
\end{align}
We apply Lemma \ref{lem3.1} (i), H\"{o}lder's inequality and the Sobolev embedding to get
\begin{align*}
\big\|\mathcal{A}^{  \frac{1}{2} }  e^{-t\mathcal{A}}\mathbf{u}(\cdot,t_3)\big\|_{L^{ 1+\gamma}(\Omega)}
 \le  C_5 e^{-\lambda_{1} (t-t_2)}
\big\|\mathcal{A}^{  \frac{1}{2} } \mathbf{u}(\cdot,t_2)\big\|_{L^{ 1+\gamma}(\Omega)},
\qquad\forall~ t>t_3
\end{align*}
and
\begin{align*}
&\quad\int_{t_3}^{t}\big\|\mathcal{A}^{  \frac{1}{2} }  e^{-(t-s)\mathcal{A}}\mathbf{u}(\cdot, s)
\cdot \nabla \mathbf{u}(\cdot, s)\big\|_{L^{1+\gamma }(\Omega)} \,ds \nonumber \\
&\le
C_6\int_{t_3}^{t}
(t-s)^{-1+\frac{1}{2(1+\gamma) } } e^{-\lambda_{1}(t-s)}
\| \mathbf{u}(\cdot, s) \cdot \nabla \mathbf{u}(\cdot, s) \|_{L^{   \frac{2+2\gamma}{2+\gamma} }(\Omega)} \,ds \nonumber \\
&\le
C_7\int_{t_3}^{t}  (t-s)^{  -1+\frac{1}{2(1+\gamma) }   } e^{-\lambda_{1}(t-s)}
\|\mathbf{u}(\cdot, s)\|_{L^{ 2(1+\gamma)  }(\Omega)} \|\nabla \mathbf{u}(\cdot, s) \|_{L^{ 2 }(\Omega)} \,ds
\nonumber \\
&\le C_8
\int_{t_3}^{t}  (t-s)^{ -1+\frac{1}{2(1+\gamma) }  }
e^{-\lambda_{1}(t-s)}\|\nabla\mathbf{u}(\cdot, s)\|^2_{L^{2}(\Omega)} \,ds,
\qquad\forall~t >t_3
\end{align*}
and
\begin{align*}
&\quad\int_{t_3}^{t}\big\|\mathcal{A}^{  \frac{1}{2} }  e^{-(t-s)\mathcal{A}} n(\cdot, s)\big\|_{L^{1+\gamma}(\Omega)}\,ds \nonumber \\
&\le
C_9\int_{t_3}^{t}
(t-s)^{-\frac{1}{2} } e^{-\lambda_{1}(t-s)}\| n(\cdot, s) \|_{L^{ 1+\gamma  }(\Omega)}\,ds \nonumber \\
&\le C_9\int_{t_3}^{t} e^{-\lambda_{1}(t-s)}
\| n(\cdot, s) \|^{1+\gamma}_{L^{ 1+\gamma  }(\Omega)}\,ds
+C_9\int_{t_3}^{t}(t-s)^{-\frac{1+\gamma}{2\gamma}} e^{-\lambda_{1}(t-s)}\,ds,\qquad\forall~t >t_3,
\end{align*}
where the positive constants $C_5,C_6,C_7,C_8,C_9$ depend only on $\gamma$ and $\Omega$.
Summarizing the above estimates and recalling \eqref{2A8}, \eqref{2A1'} deduce the existence of $t_5>t_3$ such that
\begin{align*}
\|\mathcal{A}^{  \frac{1}{2} } \mathbf{u}(\cdot, t)\|_{L^{ 1+\gamma }(\Omega)}
\le C_{10},
\qquad\forall~t >t_5,
\end{align*}
and hence,
\begin{align}\label{2A9}
\|\mathbf{u}(\cdot, t)\|_{L^{ \infty }(\Omega)}\le C_{11},
\qquad\forall~t >t_5
\end{align}
due to an application of the Sobolev embedding inequality.

 By utilizing \eqref{example2}$_1$ and noticing \eqref{cDS}, we obtain for any $t>t_6:=\max\{t_4,t_5\}$,
   \begin{align*}
   \frac{1}{2}\frac{\mathrm{d}}{\mathrm{d}t}\int_{\Omega} n^2 \,dx+\int_{\Omega}D(n)|\nabla n|^2\,dx
   &=-
    \int_{\Omega}S(n)\nabla n \cdot\nabla c \,dx+ r \int_{\Omega} n^{2}\,dx-\mu \int_{\Omega} n^{\gamma+2}\,dx\\
     &=
    \int_{\Omega}H_s(n)\cdot\Delta c\,dx + r \int_{\Omega} n^{2}\,dx -\mu \int_{\Omega} n^{\gamma+2}\,dx\\
    &\le \int_{\Omega}H_s(n)^{ \frac{\gamma+2}{\gamma+1}  }\,dx
    +\int_{\Omega} |\Delta c|^{\gamma+2}\,dx
+ r \int_{\Omega} n^{2}\,dx -\mu \int_{\Omega} n^{\gamma+2}\,dx
  \end{align*}
  with
 \begin{align*}
 H_s(n):=\int_0^n S(\sigma) \,d\sigma&\le C_{12}n^{(\beta-1)_++2}+C_{12}\\
 &\le C_{12}n^{\gamma+1}+2C_{12},
 \end{align*}
 where we utilized the assumption $\beta\le \gamma$ and $C_{12}$ depends only on $b_0$, $r$ and $\Omega$.
  A simple calculation implies
$$(2+\gamma+2 r)\int_{\Omega}n^2\,dx\leq \mu\int_{\Omega} n^{\gamma+2}\,dx
+ \frac{ \big(2+\gamma+2 r\big)^{\frac{2}{\gamma}+1}  |\Omega|}{ \mu^{\frac{2}{\gamma}} }.$$
Combining the above estimates, one has
    \begin{align*}
\frac{\mathrm{d}}{\mathrm{d}t}\int_{\Omega} n^2\,dx+(2+\gamma)\int_{\Omega}n^2\,dx
 &\le 2C_{12}^{ \frac{\gamma+2}{\gamma+1}  } \int_{\Omega}n^{ \gamma+2  }\,dx
+\int_{\Omega} |\Delta c|^{\gamma+2}\,dx
-\mu \int_{\Omega} n^{\gamma+2}\,dx\\
&\quad+\frac{ \big(2+\gamma+2 r\big)^{\frac{2}{\gamma}+1}  |\Omega|}{ \mu^{\frac{2}{\gamma}} }
+4C_{12}^{ \frac{\gamma+2}{\gamma+1}  }, \qquad\forall~t >t_6,
  \end{align*}
  namely,
\begin{align}\label{A210}
   \frac{\mathrm{d}}{\mathrm{d}t}\Big(e^{ (\gamma+2)s}\int_{\Omega} n^2 \,dx\Big)
 &  \leq
2C_{12}^{ \frac{\gamma+2}{\gamma+1}  }e^{ (\gamma+2)s}
\int_{\Omega}n^{ \gamma+2  }\,dx
+e^{ (\gamma+2)s}\int_{\Omega} |\Delta c|^{\gamma+2}\,dx
-\mu e^{ (\gamma+2)s}\int_{\Omega} n^{\gamma+2}\,dx\nonumber\\
&\quad+e^{ (\gamma+2)s}\frac{ \big(2+\gamma+2 r\big)^{\frac{2}{\gamma}+1}  |\Omega|}{ \mu^{\frac{2}{\gamma}} }
+4C_{12}^{ \frac{\gamma+2}{\gamma+1}  }e^{ (\gamma+2)s}, \qquad\forall~t >t_6.
  \end{align}
The maximal $L^p$--$L^q$ estimates \cite{Hieber} along with the Gagliardo-Nirenberg inequality and the 
Sobolev estimates indicate that
\begin{align*}
\int_{t_6}^t \int_{\Omega} e^{ (\gamma+2)s}|\Delta c|^{\gamma+2} \,dxd s
 &\leq C_{13}\int_{t_6}^t\int_{\Omega}e^{ (\gamma+2)s} |\mathbf{u}\cdot \nabla c|^{\gamma+2}\,dxds
 +C_{13} \int_{t_6}^t \int_{\Omega} e^{ (\gamma+2)s} n^{\gamma+2}\,dxds\\
&\quad +C_{13}e^{ (\gamma+2)t_6}\int_{\Omega}|\Delta c|^{\gamma+2}(x,t_6) \,dx \\
 &\leq C_{13}\|\mathbf{u}\|^{\gamma+2}_{L^{ \infty}(\Omega\times(t_6,t))}
 \int_{t_6}^te^{ (\gamma+2)s} \| \nabla c\|_{L^{ \gamma +2 }(\Omega)}^{\gamma+2}\,ds
 + C_{13}\int_{t_6}^t \int_{\Omega} e^{ (\gamma+2)s} n^{\gamma+2}\,dxds\\
&\quad +C_{13}e^{ (\gamma+2)t_6}\int_{\Omega}|\Delta c|^{\gamma+2}(x,t_6) \,dx\\
  &\leq C_{14}\|\mathbf{u}\|^{\gamma+2}_{L^{ \infty}(\Omega\times(t_6,t))}\int_{t_6}^t  e^{ (\gamma+2)s}
 \Big( \| \Delta c\|_{L^{\gamma+ 2 }(\Omega)}^{\frac{\gamma+2}{2}} \|c\|_{L^{\gamma+ 2}(\Omega)}^{\frac{\gamma+2}{2}}
 + \|c\|_{L^{\gamma+ 2}(\Omega)}^{\gamma+2  }\Big)\,ds\\
  &\quad+  C_{13}\int_{t_6}^t \int_{\Omega} e^{ (\gamma+2)s} n^{\gamma+2}\,dxds
  +C_{13}e^{ (\gamma+2)t_6}\int_{\Omega}|\Delta c|^{\gamma+2}(x,t_6) \,dx,\qquad\forall~t >t_6
\end{align*}
with $C_{13}=C_{13}(\gamma,\Omega)>0$ and $C_{14}=C_{14}(\gamma,\Omega)>0$.
Thus, by virtue of \eqref{2A9} and \eqref{2A8'}, we employ Young's inequality to the above display
and obtain
\begin{align}\label{A211}
 \int_{t_6}^t \int_{\Omega} e^{ (\gamma+2)s}|\Delta c|^{\gamma+2} \,dxds
& \leq C_{15}\int_{t_6}^t \int_{\Omega} e^{ (\gamma+2)s} n^{\gamma+2}\,dxd s 
+ C_{15} e^{ (\gamma+2)t_6}\int_{\Omega}|\Delta c|^{\gamma+2}(x,t_6) \,dx\nonumber\\
&\quad + C_{15} e^{ (\gamma+2)t},\qquad \forall~t>t_6.
\end{align}
Integrating \eqref{A210} over $(t_6,t)$ and using \eqref{A211} directly shows that
\begin{align}\label{530a4}
e^{ (\gamma+2)t}\int_{\Omega} n^2(x,t) \,dx
 &  \leq
\left(2C_{12}^{ \frac{\gamma+2}{\gamma+1}  }+ C_{15}\right)\int_{t_6}^t \int_{\Omega} e^{ (\gamma+2)s} n^{\gamma+2}\,dxd s
-\mu \int_{t_6}^t \int_{\Omega} e^{ (\gamma+2)s} n^{\gamma+2}\,dxd s\nonumber\\
&\quad+e^{ (\gamma+2)t_6}\int_{\Omega} n^2(x,t_6) \,dx
+ C_{15} e^{ (\gamma+2)t_6}\int_{\Omega}|\Delta c|^{\gamma+2}(x,t_6) \,dx\nonumber\\
&\quad+\left(\frac{ \big(2+\gamma+2 r\big)^{\frac{2}{\gamma}+1}  |\Omega|}{ \mu^{\frac{2}{\gamma}} }
+4C_{12}^{ \frac{\gamma+2}{\gamma+1}  }+C_{15}\right)e^{ (\gamma+2)t},\qquad \forall~t>t_6.
  \end{align}
We further enlarge $\mu_*>2C_{12}^{ \frac{\gamma+2}{\gamma+1}  }+ C_{15}$ (if necessary) and let $\mu\ge\mu_*$ to conclude that there is $t_7>t_6$ fulfilling
\begin{align*}
\int_{\Omega} n^2(x,t)\,dx
\leq
\frac{ \big(2+\gamma+2 r\big)^{\frac{2}{\gamma}+1}  |\Omega|}{ \mu_*^{\frac{2}{\gamma}} }
+4C_{12}^{ \frac{\gamma+2}{\gamma+1}  }+2C_{15}+1,\qquad \forall~t>t_7,
\end{align*}
where all the factors appearing on the right-hand side depend only on 
$\gamma$, $\beta$, $r$, $b_0$ and $\Omega$.
This ends our proof.
\end{proof}

\subsection*{Conflict of interest} The authors declare that there is no conflict of interest. We also declare that this
manuscript has no associated data.

\subsection*{Data availability} Data sharing is not applicable to this article as no datasets were generated or analysed
during the current study.

 \end{document}